\documentclass{amsart}

\usepackage[latin1]{inputenc}
\usepackage{amsfonts}
\usepackage{yfonts}
\usepackage{eufrak}
\usepackage{amsmath}
\usepackage{amsthm}
\usepackage{a4wide}
\usepackage[all]{xy}
\usepackage{textcomp}

\newtheorem{thm}{Theorem}[subsection]
\newtheorem{tthm}{Theorem}[section]
\theoremstyle{defi}
\newtheorem{defi}[thm]{Definition}
\theoremstyle{ddefi}
\newtheorem{ddefi}[tthm]{Definition}
\theoremstyle{prop}
\newtheorem{prop}[thm]{Proposition}
\theoremstyle{rmk}

\theoremstyle{rrmk}

\theoremstyle{lemma}
\newtheorem{lemma}[thm]{Lemma}
\theoremstyle{llemma}
\newtheorem{llemma}[tthm]{Lemma}

\theoremstyle{ex}
\newtheorem{ex}[thm]{Example}
\theoremstyle{cor}

\begin{document}

\title[Contact Homology, Capacity and Non-Squeezing via Generating Functions]
{Contact Homology, Capacity and Non-Squeezing in $\mathbb{R}^{2n}\times S^{1}$ via Generating Functions}

\author{Sheila Sandon}
\address{Departamento de Matem\'{a}tica, Instituto Superior T\'{e}cnico, Av. Rovisco Pais, 1049-001 Lisboa, Portugal}
\email{sandon@math.ist.utl.pt}

\begin{abstract}
\noindent
Starting from the work of Bhupal \cite{B}, we extend to the contact case the Viterbo capacity and Traynor's construction of symplectic homology. As an application we get a new proof of the Non-Squeezing Theorem of Eliashberg, Kim and Polterovich \cite{EKP}.
\end{abstract}

\maketitle

\section{Introduction}\label{intro}

Consider the domains $B^{2n}(R)= \{\,\pi \sum_{i=1}^n x_i^{\,2} + y_i^{\,2} <R\,\}$ and $C^{2n}(R)=B^2(R)\times\mathbb{R}^{2n-2}$ in the standard symplectic euclidean space $\big(\mathbb{R}^{2n}, \omega_0=dx\wedge dy\big)$. Gromov's Non-Squeezing Theorem \cite{Gr} states that if $R_2<R_1$ then there is no symplectic embedding of $B(R_1)$ into $C(R_2)$. The analogous statement for balls and cylinders in the standard contact euclidean space $\big(\mathbb{R}^{2n+1}, \xi_0=\text{ker}\,(dz-ydx)\big)$ is trivially false, because one can use the contact transformation $(x,y,z)\mapsto (\alpha x, \alpha y, \alpha^2 z)$, where $\alpha$ is some positive constant, to squeeze any domain into an arbitrarily small ball\footnote{\hspace{1mm}In fact, as Francisco Presas explained to me, it is even possible to find a contact embedding of the whole $\mathbb{R}^{2n+1}$ into an arbitrarily small ball. A proof of this can be found for example in \cite{CKS}.}.
However an interesting non-squeezing phenomenon arises if we consider the contact manifold $\mathbb{R}^{2n}\times S^{1}$ instead of $\mathbb{R}^{2n+1}$, and the following stronger notion of contact squeezing.

\begin{ddefi}[\cite{EKP}]
Given open domains $\mathcal{U}_1$ and $\mathcal{U}_2$ in a contact manifold $(V,\xi)$ we say that $\mathcal{U}_1$ can be squeezed into $\mathcal{U}_2$ if there exists a contact isotopy $\varphi_t: \overline{\mathcal{U}_1} \longrightarrow V$, $t\in [0,1]$, such that $\varphi_0$ is the identity and $\varphi_1(\overline{\mathcal{U}_1})\subset\mathcal{U}_2$. We say that $\mathcal{U}_1$ can be squeezed into $\mathcal{U}_2$ inside a third domain $\mathcal{V}$ if $\varphi_t(\overline{\mathcal{U}_1})\subset\mathcal{V}$ for all $t$.
\end{ddefi}

Note that if $\overline{\mathcal{U}_1}$ is compact then by the isotopy extension theorem (see for example \cite{Ge}) any contact squeezing of $\mathcal{U}_1$ into $\mathcal{U}_2$ 
inside $\mathcal{V}$ can be extended to a global contactomorphism of $V$ supported in $\mathcal{V}$.\\
\\
Given a domain $\mathcal{U}$ in $\mathbb{R}^{2n}$ we will denote by $\widehat{\mathcal{U}}$ the domain $\mathcal{U}\times S^1$ in $\mathbb{R}^{2n}\times S^{1}$. In \cite{EKP} it is proved that for any $R_1$, $R_2$ there exists a contact embedding of $\widehat{B(R_1)}$ into $\widehat{B(R_2)}$, which if $n>1$ is isotopic through smooth embeddings to the inclusion $\widehat{B(R_1)}\hookrightarrow\mathbb{R}^{2n}\times S^{1}$. However, this isotopy cannot be made contact if $R_2 < k \leq R_1$ for some integer $k$.

\begin{tthm}[Non-Squeezing Theorem \cite{EKP}]\label{thm1}
Assume $R_2\leq k \leq R_1$ for some integer $k$. Then the closure of $\widehat{B(R_1)}$ cannot be mapped into $\widehat{B(R_2)}$ by a compactly supported contactomorphism of $\mathbb{R}^{2n}\times S^1$. In particular, $\widehat{B(R_1)}$ cannot be squeezed into $\widehat{B(R_2)}$. 
\end{tthm}

Eliashberg, Kim and Polterovich also proved that $\widehat{B(R_1)}$ can be squeezed into $\widehat{B(R_2)}$ if $R_1$ and $R_2$ are smaller than $1$ and if $n>1$ (in the 3-dimensional case it is never possible to squeeze $\widehat{B(R_1)}$ into a smaller $\widehat{B(R_2)}$, as can be seen using the techniques in \cite{E}). It remains an open question whether $\widehat{B(R_1)}$ can be squeezed into $\widehat{B(R_2)}$ for $n>1$ and $k-1 < R_2 \leq R_1 <k$ with $k>1$.\\
\\
An interesting feature of contact squeezing is that it requires extra room. For example, if $R_2\leq \frac{1}{l}\leq R_1$ for some integer $l$, then any contact squeezing of $\widehat{B(R_1)}$ into $\widehat{B(R_2)}$ must move $\widehat{B(R_1)}$ outside $\widehat{B(\frac{1}{l-1})}$ at a certain time. This is a special case of the following theorem.

\begin{tthm}[\cite{EKP}]\label{thm2}
Assume that $R_2\leq\frac{k}{l}\leq R_1<R_3 \leq\frac{k}{l-1}$ for some integers $k$ and $l$. Then the closure of $\widehat{B(R_1)}$ cannot be mapped into $\widehat{B(R_2)}$ by a compactly supported contactomorphism $\psi$ of $\mathbb{R}^{2n}\times S^1$ with $\psi\,\big(\widehat{B(R_3)}\big)=\widehat{B(R_3)}$. In particular, $\widehat{B(R_1)}$ cannot be squeezed into $\widehat{B(R_2)}$ inside $\widehat{B(\frac{k}{l-1})}$.
\end{tthm}

Theorems \ref{thm1} and \ref{thm2} are proved in \cite{EKP} using contact homology of fiberwise starshaped domains in $\mathbb{R}^{2n}\times S^1$. This is a special instance of the Symplectic Field Theory, and is related to a version of the filtered symplectic homology of domains in $\mathbb{R}^{2n}$ as used in \cite{BPS}, \cite{CGK} and \cite{GG}. We will present here a proof of the same results using generating functions instead of holomorphic curves techniques.\\
\\
Generating functions have been studied extensively by many authors in the 80's and 90's. They provide a powerful tool in symplectic and contact topology, with important applications also to many of the central problems of these subjects (see for instance \cite{Ch1}, \cite{LS}, \cite{S}, \cite{S2}, \cite{G}, \cite{V}, \cite{T}, \cite{G2}, \cite{Ch},\cite{Th}, \cite{V3}, \cite{C}, \cite{EG}, \cite{B2}, \cite{Th3}, \cite{M}, \cite{Th2}, \cite{B}, \cite{T2}, \cite{CP05}, \cite{FP06}, \cite{JT06}, \cite{CFP}, \cite{CN}, \cite{CN2}, \cite{ELST08}, \cite{FR}). In particular, Viterbo \cite{V} applied Morse-theoretical methods to the generating function of a Lagrangian submanifold $L$ of the cotangent bundle of a closed manifold $B$ to define invariants $c(u,L) \in \mathbb{R}$ for any $u\in \text{H}^{\ast}(B)$. Using this he could then define an invariant $c(\phi)$ for compactly supported Hamiltonian symplectomorphisms $\phi$ of $\mathbb{R}^{2n}$, which in turn led to the definition of a symplectic capacity for domains in $\mathbb{R}^{2n}$. Among the applications discussed by Viterbo there is in particular the definition of a partial order and a bi-invariant metric on the group of compactly supported Hamiltonian symplectomorphisms of $\mathbb{R}^{2n}$.\\
\\
Extending the work of Viterbo, Traynor \cite{T} defined homology groups for Hamiltonian symplectomorphisms and, via a limit process, domains of $\mathbb{R}^{2n}$. More precisely, for any domain $\mathcal{U}$ in $\mathbb{R}^{2n}$ and any interval $(a,b]$ of $\mathbb{R}$ she defined homology groups $G_{\ast}^{\;\;(a,b]}\,(\mathcal{U})$. She proved that these groups are symplectic invariants and calculated them in the case of open ellipsoids.\\
\\
Some of the above results have been extended to contact topology. In particular, Bhupal \cite{B} defined invariants $c\,(u,L)$ for a Legendrian submanifold $L$ of the 1-jet bundle of a closed manifold $B$ and $u\in \text{H}^{\ast}(B)$. Proceeding as in \cite{V} he then associated a number $c(\phi)$ to each compactly supported contactomorphisms $\phi$ of $\mathbb{R}^{2n+1}$ isotopic to the identity, and used this construction to define a partial order on the groups of all such contactomorphisms. In contrast with the symplectic case, the number $c(\phi)$ is not invariant by conjugation of $\phi$ with another contactomorphism $\psi$. For this reason it is not possible to mimic Viterbo's construction of a symplectic capacity to obtain a contact invariant for domains in $\mathbb{R}^{2n+1}$. However Bhupal could prove that $c(\psi\phi\psi^{-1})=0$ if and only if $c(\phi)=0$, which was all he needed to define the partial order. Our contribution to this problem is the observation that if we consider contactomorphisms of $\mathbb{R}^{2n}\times S^1$, regarded as contactomorphisms of $\mathbb{R}^{2n+1}$ that are 1-periodic in the $z$-coordinate, then the methods of Bhupal can be used to show that $c(\psi\phi\psi^{-1})=k$ if and only if $c(\phi)=k$, where $k$ is any positive integer. In particular this implies that the integer part of $c(\phi)$ is invariant by conjugation, and this fact can be used to define an integral contact invariant for domains in $\mathbb{R}^{2n}\times S^1$. In analogy with the symplectic case we call this invariant a contact capacity. Given a domain $\mathcal{U}$ in $\mathbb{R}^{2n}$, we prove that the contact capacity of $\widehat{\mathcal{U}}$ equals the integer part of the Viterbo capacity of $\mathcal{U}$. This then easily yields a proof of Theorem \ref{thm1} (see \ref{contact_cap}).\\
\\
Similar observations can be made about homology groups. Using the set-up of Bhupal, it is possible to extend the construction of Traynor to the contact case and get homology groups $G_{\ast}^{\;\;(a,b]}\,(\mathcal{V})$ for a domain $\mathcal{V}$ of $\mathbb{R}^{2n+1}$. These groups however are not contact invariant, but they become so in the 1-periodic case if we consider only integer values of $a$ and $b$.\\
\\
The crucial fact that explains the special role played by the integers in the contact 1-periodic case is the following. In the symplectic case there is a 1-1 correspondence between critical points of the generating function of a Hamiltonian symplectomorphism $\phi$ and the fixed points of $\phi$. Moreover, critical values are given by the symplectic action of the corresponding fixed points. Since the symplectic action is invariant by conjugation it follows that the generating functions of $\phi$ and of $\psi\phi\psi^{-1}$ have the same critical values. This fundamental fact can be used to prove that Viterbo capacity and Traynor's homology groups are symplectic invariants (see \ref{Vit_capacity} and \ref{sympl_hom}). The same argument does not apply to the contact case. Given a contactomorphism $\phi$ of $\mathbb{R}^{2n+1}$ we will see in \ref{contactomorphisms} that critical points of the generating function of $\phi$ with critical value $c$ correspond to points $(x,y,z)$ of $\mathbb{R}^{2n+1}$ such that $\phi(x,y,z)=(x,y,z+c)$. Thus the generating functions of $\phi$ and of $\psi\phi\psi^{-1}$ do not have the same critical values in general. However, if one of the two functions has $0$ as critical value then so does the other as well, because critical points with critical value $0$ correspond to fixed points. Similarly, in the 1-periodic case the same holds if we replace $0$ by any integer $k$. We will explain in \ref{contact_cap} and \ref{2.3} how this observation implies that our homology groups and integral capacity for domains of $\mathbb{R}^{2n}\times S^1$
are contact invariants.\\
\\
We will now show how one can use our construction of contact homology to prove Theorems \ref{thm1} and \ref{thm2}, referring to \ref{2.3} for all technical details.\\
\\
Assume we have $R_1$, $R_2$, $R_3$ with $R_2\leq\frac{k}{l}< R_1<R_3\leq\frac{k}{l-1}$. We have to show that $\widehat{B(R_1)}$ cannot be mapped into $\widehat{B(R_2)}$ by a contactomorphism $\psi$ of $\mathbb{R}^{2n}\times S^1$ such that $\psi\,\big(\widehat{B(R_3)}\big)=\widehat{B(R_3)}$. Suppose this can be done. Then we can consider the following commutative diagram:

\begin{displaymath}
\xymatrix{
 G_{\ast}^{\;\;(k,\infty]}\,(\widehat{B(R_3)}) \ar[r] &
 G_{\ast}^{\;\;(k,\infty]}\,(\widehat{B(R_1)})  \\
 G_{\ast}^{\;\;(k,\infty]}\,(\widehat{B(R_3)}) \ar[u]^{\psi_{\ast}} \ar[r] & 
 G_{\ast}^{\;\;(k,\infty]}\,(\widehat{B(R_2)}) \ar[r] & 
 G_{\ast}^{\;\;(k,\infty]}\,\big(\psi(\widehat{B(R_1)})\big) \ar[ul]_{\psi_{\ast}}}
\end{displaymath}
where the horizontal arrows denote the homomorphisms induced by inclusions (see Theorem \ref{monot}) and the vertical ones are isomorphisms induced by $\psi$ (see Theorem \ref{cont_inv}). Consider $\mathbb{Z}_2$-coefficients, and $\ast=2nl$. Then by Theorems \ref{thmballs} and \ref{lift3} we know that $G_{\ast}^{\;\;(k,\infty]}\,(\widehat{B(R_2)})=0$, 
$G_{\ast}^{\;\;(k,\infty]}\,(\widehat{B(R_1)})=G_{\ast}^{\;\;(k,\infty]}\,(\widehat{B(R_3)})=\mathbb{Z}_2$, and that the horizontal map on the top is an isomorphism. Thus the diagram gives a contradiction, yielding the proof of Theorem \ref{thm2}. Theorem \ref{thm1} can be proved similarly, considering $\ast=2n$ and a big enough $R_3$.\\
\\
This article is organized as follows.\\
\\
In Section \ref{sympl} we describe the constructions by Viterbo and Traynor of a symplectic capacity and symplectic homology for domains in $\mathbb{R}^{2n}$. In \ref{sympl_hom} we define homology groups for compactly supported Hamiltonian symplectomorphisms of $\mathbb{R}^{2n}$ and use them to construct, via a limit process, symplectic homology of domains. The limit process is based on the Viterbo partial order on $\text{Ham}^c\,(\mathbb{R}^{2n})$, which we discuss in \ref{Vitorder}. The Viterbo capacity is described in \ref{Vit_capacity}. The partial order and the capacity are defined using the invariants for Hamiltonian symplectomorphisms introduced by Viterbo. We discuss these invariants in \ref{Vit_inv} and \ref{inv_symplectomorphisms}. In \ref{lagrsubmfds} and \ref{symplectomorphisms} we give the needed preliminaries on generating functions. In this section we always follow \cite{T} and \cite{V} except for the following points: we give a different proof of symplectic invariance of the homology groups (Proposition \ref{conjsympl}); monotonicity of the invariant $c(\phi)$ is proved directly in [Vit92, Proposition 4.6] while for us is an immediate consequence of Proposition \ref{parordsymp}.\\
\\
In Section \ref{contact} we generalize the results of Section \ref{sympl} to the contact case. In \ref{contact_cap} and \ref{2.3} respectively we construct a contact capacity and contact homology groups for domains in  $\mathbb{R}^{2n}\times S^1$. The limit process to define contact homology of domains is based on the Bhupal partial order on the group of contactomorphisms of $\mathbb{R}^{2n+1}$, which we discuss in \ref{Bhupalorder}. All the constructions in this section use the generalization of the Viterbo's invariants to contactomorphisms of $\mathbb{R}^{2n+1}$ and $\mathbb{R}^{2n}\times S^1$. We discuss these invariants in \ref{s_inv_leg} and \ref{inv_contactomorphisms}. In \ref{legendrian} and \ref{contactomorphisms} we give respectively some preliminaries on generating functions in contact topology, and a more detailed discussion of generating functions for contactomorphisms of $\mathbb{R}^{2n+1}$.

\subsection*{Acknowledgements}
I am very grateful to Lisa Traynor for encouraging this project from the very beginning, as well as to Josh Sabloff for his interest in my work. Mohan Bhupal and David Th\'{e}ret kindly sent me a copy of their theses; both of them have been very important to develop this work. I have been learning a lot in discussions with Francisco Presas. In particular, I thank him for a crucial suggestion that dramatically improved my understanding of the most important point of this article. I also thank David Martinez Torres for helping me with my problems with Morse theory, and the anonymous referee for a very helpful review. During the preparation of this article I had the opportunity to talk about my research with Yakov Eliashberg and Leonid Polterovich. I thank them both for patiently listening to me, and for enlightening (when not obscure for me) comments. I also thank Emmanuel Giroux for writing about my work in his Bourbaki talk. This has been not only a great honour for me but also a big help, since every time I started doubting about what I wrote I could always turn to his paper and believe everything again. Finally, and most of all, I am especially gratefull to my Ph.D. supervisor Miguel Abreu. All the progress I made since I am in Lisbon, as a person as well as an apprentice mathematician, would never have been possible without his sensitive help and steady support.\\
\\
My research was supported by an FCT graduate fellowship, program POCTI-Research Units Pluriannual Funding Program through the Center for Mathematical Analysis Geometry and Dynamical Systems and Portugal/Spain cooperation
grant FCT/CSIC-14/CSIC/08.

\section{Symplectic Capacity and Homology for Domains in $\mathbb{R}^{2n}$}\label{sympl}

We refer to \cite{MS} for preliminaries on symplectic topology. Here we only discuss some basic concepts that are needed for the rest of the article.\\
\\
A symplectic manifold is an even dimensional manifold $M$ endowed with a symplectic form, i.e a non-degenerate closed 2-form $\omega\in\Omega^2(M)$. A symplectic manifold $(M,\omega)$ is said to be exact if $\omega=-d\lambda$ for some 1-form $\lambda$ which is then called a Liouville form. In this paper we will only deal with the following two (exact) symplectic manifolds: the standard symplectic euclidean space 
$\big(\mathbb{R}^{2n},\, \omega_0=-d\,(ydx)\big)$ and the cotangent bundle $T^{\ast}B$ of a manifold $B$, endowed with the canonical symplectic form $\omega_{\text{can}}=-d\,(pdq)$ where $q$ is the coordinate on the base and $p$ on the fiber. A diffeomorphism $\phi$ of a symplectic manifold $(M,\omega)$ is called a symplectomorphism if $\phi^{\ast}\omega=\omega$. Given a time-dependent function $H_t$ on $M$, the flow $\phi_t$ of the time-dependent vector field $X_t$ defined by the condition $i_{X_t}\omega=-dH_t$ consists of symplectomorphisms. The isotopy $\phi_t$ is called a Hamiltonian isotopy, with Hamiltonian function $H_t$. A Hamiltonian symplectomorphism of $(M,\omega)$ is a symplectomorphism that can be obtained as the time-1 map of a Hamiltonian isotopy. An immersion $i:L\rightarrow (M,\omega)$ is called isotropic if $i^{\ast}\omega=0$ and Lagrangian if moreover the dimension of $L$ is maximal, i.e. half of the dimension of $M$. If $(M,\omega)$ is exact with Liouville form $\lambda$, then a Lagrangian immersion $i:L\rightarrow (M,\omega)$ is called exact if $i^{\ast}\lambda=df$ for some function $f$.\\
\\
Consider an exact symplectic manifold $(M,\omega=-d\lambda)$. The \textbf{action functional} $\mathcal{A}_H$ with respect to a time-dependent Hamiltonian $H$ is defined by 
$$\mathcal{A}_H(\gamma):=\int_{t_0}^{t_1}\Big(\lambda\big(\dot{\gamma}(t)\big)+H_t\,\big(\gamma(t)\big)\Big)\:dt$$
for a path $\gamma: [t_0,t_1]\rightarrow M$.
A crucial fact is that $\gamma$ is a critical point of $\mathcal{A}_H$ (with respect to variations with fixed endpoints) if and only if it is an integral curve of the Hamiltonian flow of $H$. Moreover we have the following lemma.

\begin{llemma}[\cite{MS}, 9.19]\label{hamis}
Let $\phi_t$, $t\in [0,1]$, be a symplectic isotopy of an exact symplectic manifold $\big(M,\omega=-d\lambda\big)$, starting at the identity. Then $\phi_t$ is a Hamiltonian isotopy if and only if $\phi_t^{\;\ast}\lambda-\lambda=dF_t$ for a smooth family of functions $F_t:M\longrightarrow\mathbb{R}$. In this case the $F_t$ are given by
$$F_t=\int_0^t\big(\lambda(X_s)+H_s\big)\circ \phi_s\:ds $$
where $X_t$ is the vector field generating $\phi_t$, and $H_t:M\longrightarrow\mathbb{R}$ the corresponding Hamiltonian function. In other words, the value of $F_t$ at a point $q$ of $M$ is the action functional with respect to $H$ of the path $\phi_s(q)$, $s\in [0,t]$.
\end{llemma}

The action functional plays a central role in symplectic topology. It is also related in a crucial way to generating functions and thus to the invariants defined by Traynor and Viterbo that we are going to discuss in this section.

\subsection{Generating functions for Lagrangian submanifolds of $T^{\ast}B$}\label{lagrsubmfds}

Consider a smooth manifold $B$. Given a function $f:B\rightarrow \mathbb{R}$, the graph of its differential is a Lagrangian submanifold $L_f$ of $T^{\ast}B$. Many geometric properties of $L_f$ can be inferred by looking at $f$, the most immediate instance of this being the fact that critical points of $f$ correspond to intersection points of $L_f$ with the 0-section. The idea of generating functions is to generalize this construction in order to be able to  associate a function to a more general class of Lagrangian submanifolds of $T^{\ast}B$. This can be achieved by considering functions defined on the total space of a fiber bundle over $B$, and by using the following construction  which is due to H\"{o}rmander.

\begin{defi}[\cite{H}]
A \emph{variational family} $(E,S)$ over a manifold $B$ is a function $S:E \longrightarrow \mathbb{R}$ defined on the total space of a fiber bundle $p:E \longrightarrow B $. $(E,S)$ is a \emph{transverse variational family} if $dS: E\longrightarrow T^{\ast}E$ is transverse to $N_E:=\{\:(e,\eta)\in T^{\ast}E \;|\; \eta \equiv 0 \;\emph{on} \;\emph{ker}\;\emph{d}p\,(e)\:\}$.
\end{defi}

Consider the set $\Sigma_S$ of \textit{fiber critical points} of $S$, i.e. points $e$ of $E$ that are critical points of the restriction of $S$ to the fiber through $e$. Note that $\Sigma_S=(dS)^{-1}(N_E)$, so if the variational family $(E,S)$ is transverse then $\Sigma_S$ is a submanifold of $E$ of dimension equal to the dimension of $B$. To any $e$ in $\Sigma_S$ we can associate an element $v^{\ast}(e)$ of $T^{\phantom{p}\ast}_{p(e)}B$ (the \textit{Lagrange multiplier}) defined by $v^{\ast}(e)\,(X):=dS\,(\widehat{X})$ for $X \in T_{p(e)}B$, where $\widehat{X}$ is any vector in $T_eE$ with $p_{\ast}(\widehat{X})=X$.

\begin{prop}\label{horm}
If $(E,S)$ is a transverse variational family over $B$, then $i_S:\Sigma_S\longrightarrow T^{\ast}B$, $e\mapsto \big(p(e),v^{\ast}(e)\big)$ is a Lagrangian immersion.
\end{prop}

In this case, $S:E \longrightarrow \mathbb{R}$ is called a \textbf{generating function} for the Lagrangian submanifold $L_S:=i_S\,(\Sigma_S)$ of $T^{\ast}B$. 
Note that (non-degenerate) critical points of $S$ correspond under $i_S$ to (transverse) intersection points of $L_S$ with the 0-section. Note also that $i_S$ is an exact Lagrangian immersion, with $ i_S^{\phantom{S}\ast}\, \lambda_{\text{can}}=d\,(S_{|\Sigma_S})$. A proof of Proposition \ref{horm} can be found for instance in \cite[9.34]{MS}.\\
\\
A crucial example of this construction is given by the case in which $E$ is the space of paths $\gamma: [0,1]\rightarrow T^{\ast}B$ that begin at the 0-section. $E$ can be seen as a fiber bundle over $B$ with projection $\gamma \mapsto \pi\,\big(\gamma(1)\big)$, where $\pi$ is the projection of $T^{\ast}B$ into $B$. Given a time-dependent Hamiltonian $H$ on $T^{\ast}B$ we can define a function $S:E\rightarrow \mathbb{R}$ by $S(\gamma):=\mathcal{A}_H(\gamma)$. Then $\Sigma_S$ is the set of orbits of the Hamiltonian flow of $H$ and the Lagrange multiplier of an element $\gamma$ of $\Sigma_S$ is the vertical component of $\gamma(1)$. Thus $S$ generates
the image of the 0-section under the time-1 map of the Hamiltonian flow of $H$. Note that $S$ is not a generating function in the sense of the above definition because $E$ has infinite dimensional fibers. However, it is possible to approximate $E$ by a finite dimensional space and prove in this way that any Lagrangian submanifold of $T^{\ast}B$ which is Hamiltonian isotopic to the 0-section has a generating function. This was done by Sikorav, using the broken Hamiltonian trajectories idea of \cite{Ch1} and \cite{LS}. It was also proved that by this construction one can obtain in fact a generating function which is quadratic at infinity in the following sense.

\begin{defi}
A generating function $S:E \longrightarrow \mathbb{R}$ for a Lagrangian submanifold of $T^{\ast}B$ is \textbf{quadratic at infinity} if $p:E \longrightarrow B $ is a vector bundle and if there exists a non-degenerate quadratic form $Q_{\infty}: E \longrightarrow \mathbb{R}$ such that $dS-\partial_vQ_{\infty}: E \longrightarrow E^{\ast}$ is bounded, where $\partial_v$ denotes the fiber derivative.
\end{defi}

This condition is important because it makes possible to apply to generating functions all arguments of Morse theory, even though the functions are not defined on a compact manifold.

\begin{thm}[\cite{S}, \cite{S2}]\label{existence}
If $B$ is closed, then any Lagrangian submanifold of $T^{\ast}B$ which is Hamiltonian isotopic to the 0-section has a generating function quadratic at infinity (g.f.q.i.). More generally, if $L\subset T^{\ast}B$ has a g.f.q.i. and $\psi_t$ is a Hamiltonian isotopy of $T^{\ast}B$, then there exists a continuous family of g.f.q.i. $S_t:E\longrightarrow\mathbb{R}$ such that each $S_t$ generates the corresponding $\psi_t(L)$.
\end{thm}

A second fundamental result is the uniqueness theorem of Viterbo and Th\'{e}ret, which says that all generating functions of a Lagrangian submanifold of $T^{\ast}B$ which is Hamiltonian isotopic to the 0-section are related by some basic operations that do not affect the Morse theory of the function. As a consequence, all the invariants defined using generating functions do not depend on the choice of the specific generating function used to calculate them.

\begin{thm}[\cite{V}, \cite{Th2}]\label{uniqueness}
Suppose that $B$ is closed, and let $L$ be a Lagrangian submanifold of $T^{\ast}B$ Hamiltonian isotopic to the 0-section. If $S:E \longrightarrow \mathbb{R}$ is a g.f.q.i. for $L$ then any other g.f.q.i. $S'$ for $L$ can be obtained from $S$ by the following operations:
\vspace{-0.2cm}
\begin{itemize}
\item addition of a constant: $S'=S+c:E \longrightarrow \mathbb{R}$, for some $c\in \mathbb{R}$;
\item fiber-preserving diffeomorphism: $S'=S\circ\phi$, for some fiber-preserving diffeomorphism $\phi:E'\longrightarrow E$;
\item stabilization (assuming $p:E \longrightarrow B $ is a vector bundle): $S'=S+Q: E'=E\oplus F\longrightarrow \mathbb{R}$, where $F\longrightarrow B$ is a vector bundle and $Q:F\longrightarrow\mathbb{R}$ is a non-degenerate quadratic form.
\end{itemize}
\end{thm}

A g.f.q.i. $S:E\longrightarrow\mathbb{R}$ is said to be \textit{special} if $E=B\times \mathbb{R}^N$ and $S=S_0+Q_{\infty}$, where $S_0$ is compactly supported and $Q_{\infty}$ is the same quadratic form on each fiber.

\begin{prop}[\cite{Th2}]
If $B$ is closed, then any g.f.q.i. can be modified to a special one by applying the operations in Theorem \ref{uniqueness}.
\end{prop}

In the following we will always consider generating functions  which are quadratic at infinity, and we will assume that they are special whenever this is needed.

\subsection{Generating functions for Hamiltonian symplectomorphisms of $\mathbb{R}^{2n}$}\label{symplectomorphisms}

We will now apply the results of \ref{lagrsubmfds} to compactly supported Hamiltonian symplectomorphisms of $\mathbb{R}^{2n}$. We do this by associating to such a symplectomorphism $\phi$ of $\mathbb{R}^{2n}$ a Lagrangian submanifold of $T^{\ast}S^{2n}$, as we will now explain. We first drop the condition of $\phi$ being compactly supported, and construct a Lagrangian submanifold $\Gamma_{\phi}$ of $T^{\ast}\mathbb{R}^{2n}$. Note first that the graph of $\phi$ can be seen as a Lagrangian embedding 
$\text{gr}_{\phi}:\mathbb{R}^{2n}\longrightarrow\overline{\mathbb{R}^{2n}}\times\mathbb{R}^{2n}$, where 
$\overline{\mathbb{R}^{2n}}$ denotes the symplectic manifold $(\mathbb{R}^{2n},-\omega_0)$. We identify $\overline{\mathbb{R}^{2n}}\times\mathbb{R}^{2n}$ with $T^{\ast}\mathbb{R}^{2n}$ by the symplectomorphism\footnote{\hspace{1mm}One can use in fact any other symplectomorphism that sends the diagonal to the 0-section. Traynor and Viterbo use respectively $\tau':(x,y, X, Y)\mapsto (y,X, x-X,Y-y)$ and $\tau'':(x,y, X, Y)\mapsto (\frac{x+X}{2},\frac{y+Y}{2}, Y-y, x-X)$. We use $\tau$ because it is consistent with the formula in the contact case given by Bhupal (see \ref{contactomorphisms}).} $\tau:(x,y, X, Y)\mapsto (x,Y, Y-y,x-X)$ and define $\Gamma_{\phi}: \mathbb{R}^{2n}\longrightarrow T^{\ast}\mathbb{R}^{2n}$ by 
$\Gamma_{\phi}=\tau\circ\text{gr}_{\phi}$. Since $\tau$ sends the diagonal of $\overline{\mathbb{R}^{2n}}\times\mathbb{R}^{2n}$ to the 0-section of $T^{\ast}\mathbb{R}^{2n}$, fixed points of $\phi$ correspond to intersection points of $\Gamma_{\phi}$ with the 0-section. Note that $\Gamma_{\phi}$ can also be written as $\Gamma_{\phi}=\Psi_{\phi}\,(\text{0-section})$ where $\Psi_{\phi}$ is the symplectomorphism of $T^{\ast}\mathbb{R}^{2n}$ defined by the diagram
\begin{displaymath}
\xymatrix{
 \quad \overline{\mathbb{R}^{2n}}\times\mathbb{R}^{2n}\quad  \ar[r]^{\text{id}\times\phi} \ar[d]_{\tau} &
 \quad \overline{\mathbb{R}^{2n}}\times\mathbb{R}^{2n} \quad \ar[d]^{\tau} \\
 \quad T^{\ast}\mathbb{R}^{2n} \quad \ar[r]_{\Psi_{\phi}} &  \quad T^{\ast}\mathbb{R}^{2n}.}\quad
\end{displaymath}
This shows in particular that $\Gamma_{\phi}$ is Hamiltonian isotopic to the 0-section. Observe that the above diagram behaves well with respect to composition: for all Hamiltonian symplectomorphisms $\phi$, $\phi_1$ and $\phi_2$ we have namely that $\Psi_{\phi_1}\circ\Psi_{\phi_2}=\Psi_{\phi_1\phi_2} $ (in particular $\Gamma_{\phi_1\,\circ\,\phi_2}=\Psi_{\phi_1}\,(\Gamma_{\phi_2})$) and $\Psi_{\phi}^{\phantom{\phi}-1}=\Psi_{\phi^{-1}}$. Note moreover that $\Gamma_{\phi}$ is in fact an exact Lagrangian embedding, with $$\Gamma_{\phi}^{\phantom{\phi}\ast}\,(\lambda_{\text{can}})=d\,(x\phi_2-\phi_1\phi_2+F)$$ where $\phi_1$ and $\phi_2$ denote the first and last $n$ components of $\phi$ and $F$ is a function satisfying $\phi^{\ast}(\lambda_0)-\lambda_0=dF$ for $\lambda_0=ydx$ (see Lemma \ref{hamis}).\\
\\
Assume now that $\phi$ is compactly supported. Then $\Gamma_{\phi}$ coincides with the 0-section outside a compact set, so (by regarding $S^{2n}$ as the 1-point compactification of $\mathbb{R}^{2n}$) it can be seen as Lagrangian submanifold $T^{\ast}S^{2n}$, Hamiltonian isotopic to the 0-section. By Theorems \ref{existence} and \ref{uniqueness} it follows that $\Gamma_{\phi}$ has a g.f.q.i. $S:E\longrightarrow\mathbb{R}$, which is unique up to addition of a constant, fiber-preserving diffeomorphism and stabilization. We may and will always assume that $S$ is special. Note that this assumption in particular normalizes $S$, removing the indeterminacy by a constant.\\
\\
A crucial property of any generating function of a Hamiltonian symplectomorphism $\phi$ of $\mathbb{R}^{2n}$ is that its set of critical values coincides with the action spectrum of $\phi$.
\begin{defi}
Let $\phi$ be a Hamiltonian symplectomorphism of $\mathbb{R}^{2n}$. The \textbf{symplectic action} of a fixed point $q$ of $\phi$ is defined by
$$\mathcal{A}_{\phi}(q):=\mathcal{A}_H\,\big(\phi_t(q)\big)= \int_0^1\big(\lambda(X_t)+H_t\big)\,\big(\phi_t(q)\big)\:dt$$
where $\phi_t$ is a Hamiltonian isotopy joining $\phi$ to the identity, $X_t$ the vector field generating it and $H_t$ the corresponding Hamiltonian. The \emph{action spectrum} of $\phi$ is the set $\Lambda(\phi)$ of all values of $\mathcal{A}_{\phi}$ at fixed points of $\phi$.
\end{defi}

Let $F:\mathbb{R}^{2n}\rightarrow\mathbb{R}$ be the compactly supported function satisfying $\phi^{\ast}\lambda_0-\lambda_0=dF$. Then by Lemma \ref{hamis} we have $\mathcal{A}_{\phi}(q)=F(q)$, so in particular we see that the definition of $\mathcal{A}_{\phi}(q)$ does not depend on the choice of the Hamiltonian isotopy $\phi_t$ connecting $\phi$ to the identity. 

\begin{lemma}\label{crucialsympl}
Let $\phi$ be a compactly supported Hamiltonian symplectomorphism of $\mathbb{R}^{2n}$, with g.f.q.i. $S$. Then a point $q$ of $\mathbb{R}^{2n}$ is a fixed point of $\phi$ if and only if $(q,0)\in\Gamma_{\phi}$, and thus if and only if $i_S^{\;-1}(q,0)$ is a critical point of $S$. In this case the corresponding critical value is the symplectic action $\mathcal{A}_{\phi}(q)$. 
\end{lemma}
\begin{proof}
The first statement is immediate. Suppose now that we have a fixed point $q$ of $\phi$, and take a point $p$ in $\mathbb{R}^{2n}$ outside the support of $\phi$. We claim that
$$S\,\big(i_S^{\phantom{S}-1}(q,0)\big)=-\int_{\gamma\sqcup\phi(\gamma)^{-1}}\lambda_0=\mathcal{A}_{\phi}(q)$$
where $\gamma$ is any path in $\mathbb{R}^{2n}$ joining $p$ to $q$. The second equality is proved in \cite[9.30]{MS}. As for the first, it can be seen as follows. Note that
$$-\int_{\gamma\sqcup\phi(\gamma)^{-1}}\lambda_0=\int_{\gamma\times\phi(\gamma)}(-\lambda_0)\times\lambda_0$$
where $(-\lambda_0)\times\lambda_0$ is the Liouville form of $\overline{\mathbb{R}^{2n}}\times\mathbb{R}^{2n}$ and $\gamma\times\phi(\gamma)$ a path in the Lagrangian submanifold $\text{gr}_{\phi}$ of $\overline{\mathbb{R}^{2n}}\times\mathbb{R}^{2n}$. After identifying $\overline{\mathbb{R}^{2n}}\times\mathbb{R}^{2n}$ with $T^{\ast}\mathbb{R}^{2n}$ the result will follow from the following more general fact. Suppose that a Lagrangian submanifold $L$ of $T^{\ast}B$ is generated by $S:E\rightarrow \mathbb{R}$, i.e. $L$ is the image of $i_S: \Sigma_S\rightarrow T^{\ast}B$. Since $ i_S^{\phantom{S}\ast} \lambda_{\text{can}}=d\,(S_{|\Sigma_S})$ we have that  $\int_{\gamma}\lambda_{\text{can}}=S\,\big(i_S^{\phantom{S}-1}(y)\big)-S\,\big(i_S^{\phantom{S}-1}(x)\big)$ for any path $\gamma$ in $L$ joining two points $x$ and $y$. In our situation this gives
$$-\int_{\gamma\sqcup\phi(\gamma)^{-1}}\lambda_0=
\int_{\gamma\times\phi(\gamma)}(-\lambda_0)\times\lambda_0=
\int_{\tau\big(\gamma\times\phi(\gamma)\big)}\lambda_{\text{can}}$$
$$=S\,\big(i_S^{\phantom{S}-1}(q,0)\big)-S\,\big(i_S^{\phantom{S}-1}(p,0)\big)=S\,\big(i_S^{\phantom{S}-1}(q,0)\big).$$
The last equality holds because $S\,\big(i_S^{\phantom{S}-1}(p,0)\big)=0$, since $p$ is outside the support of $\phi$. The second follows from $\tau^{\ast}\lambda_{\text{can}}=(-\lambda_0)\times\lambda_0+d(-XY+xY)$ and the fact that the function $-XY+xY$ vanishes at the endpoints $(p,p)$ and $(q,q)$ of the path $\gamma\times\phi(\gamma)$.
\end{proof}

In \ref{inv_symplectomorphisms} and \ref{sympl_hom} respectively we are going to associate to any compactly supported Hamiltonian symplectomorphism $\phi$ of $\mathbb{R}^{2n}$ a real number $c(\phi)$ and, for real parameters $a$ and $b$, homology groups $G_k^{\;\;(a,b]}\,(\phi)$. The number $c(\phi)$ is obtained by selecting a critical value of the generating function $S$ of $\phi$, while the groups $G_{\ast}^{\;\;(a,b]}\,(\phi)$ are defined to be the relative homology of sublevel sets of $S$ at $a$ and $b$. Both $c(\phi)$ and $G_{\ast}^{\;\;(a,b]}\,(\phi)$ are invariant by conjugation of $\phi$ with another Hamiltonian symplectomorphism of $\mathbb{R}^{2n}$. As we will see, this is an immediate consequence of Lemma \ref{crucialsympl} and the fact that the action spectrum of a Hamiltonian symplectomorphism is invariant by conjugation. In \ref{sympl_hom} we will then apply a limit process in order to associate to any domain $\mathcal{U}$ of $\mathbb{R}^{2n}$ symplectic homology groups
$G_{\ast}^{\;\;(a,b]}\,(\mathcal{U})$, by looking at the corresponding groups for Hamiltonian symplectomorphisms supported in $\mathcal{U}$. The limit process will be with respect to the following partial order $\leq$ on the group
$\text{Ham}^c\,(\mathbb{R}^{2n})$ of compactly supported Hamiltonian symplectomorphisms of $\mathbb{R}^{2n}$: we say that $\phi_1\leq\phi_2$ if $\phi_2\phi_1^{-1}$ is the time-1 flow of some non-negative Hamiltonian (Hamiltonian functions of compactly supported symplectomorphism are normalized to be 0 outside the support). The fact that $\leq$ is indeed a partial order, in particular that if $\phi_1\leq\phi_2$ and $\phi_2\leq\phi_1$ then $\phi_1=\phi_2$, will be proved in \ref{Vitorder} by comparing $\leq$ with the partial order on $\text{Ham}^c\,(\mathbb{R}^{2n})$ defined in \cite{V}. We will need the following proposition.

\begin{prop}\label{parordsymp}
If $\phi_1\leq\phi_2$, then there are generating functions $S_1$, $S_2: E \longrightarrow \mathbb{R}$ for $\Gamma_{\phi_1}$, $\Gamma_{\phi_2}$ respectively such that $S_1\leq S_2$.
\end{prop}

\noindent
This proposition is proved in \cite[5.3]{T}. It will also follow as a special case of the corresponding result in contact geometry, that we will prove in \ref{contactomorphisms}.

\subsection{Invariants for Lagrangian submanifolds}\label{Vit_inv}

In the next four sections we will follow \cite{V} very closely. We will first define invariants for Lagrangian submanifolds of $T^{\ast}B$ and discuss their properties. Then we will apply these invariants to compactly supported Hamiltonian symplectomorphisms of $\mathbb{R}^{2n}$, and use them to define a partial order $\leq_V$ on $\text{Ham}^c\,(\mathbb{R}^{2n})$ and a capacity for domains in $\mathbb{R}^{2n}$.\\
\\
Let $B$ be a closed manifold and fix a point $P$ on it. Denote by $0_B$ the 0-section of $T^{\ast}B$ and by $\mathcal{L}_P$ the set of all Lagrangian submanifolds of $T^{\ast}B$ which are Hamiltonian isotopic to $0_B$ and such that $P\in L \cap 0_B$. We normalize generating functions by requiring that the critical point corresponding to $P$ has critical value $0$. In this way the set of critical values of a generating function for a Lagrangian submanifold $L$ depends only on $L$, and not on the choice of the generating function. Given $L$ in $\mathcal{L}_P$, we will now explain how to use a cohomology class $u$ of $B$ to select a critical value of the generating function of $L$, in order to define an invariant
$c(u,L)$.\\
\\
Let $L$ be an element of $\mathcal{L}_P$ with g.f.q.i. $S=S_0+Q_{\infty}:E\longrightarrow \mathbb{R}$. We denote by $E^{a}$, for $a\in \mathbb{R}\cup\infty$, the sublevel set of $S$ at $a$, i.e. $E^a=\{\,x\in E\,|\,S(x)\leq a\,\}$, and by $E^{-\infty}$ the set $E^{-a}$ for $a$ big (note that up to homotopy equivalence $E^{-\infty}$ is the same for all $L$ in $\mathcal{L}_P$). We will study the inclusion $i_a: (E^a,E^{-\infty})\hookrightarrow (E,E^{-\infty})$, and the induced map on cohomology
$$ i_a^{\phantom{a}\ast}: H^{\ast}(B)\equiv H^{\ast}(E,E^{-\infty})\longrightarrow H^{\ast}(E^a,E^{-\infty}).$$
Here $H^{\ast}(B)$ is identified with $H^{\ast}(E,E^{-\infty})$ via the Thom isomorphism
$$ T: H^{\ast}(B)\xrightarrow{\cong}H^{\ast}\big(D(E^-), S(E^-)\big)$$
where $E^-$ denotes the subbundle of $E$ where $Q_{\infty}$ is negative definite. Note that this isomorphism shifts the grading by the index of $Q_{\infty}$. Note also that by excision $H^{\ast}\big(D(E^-), S(E^-)\big)$ is isomorphic to $H^{\ast}(E,E^{-\infty})$. For $|a|$ big enough we have $H^{\ast}(E^a,E^{-\infty})\equiv 0$ if $a<0$, and $i_a^{\phantom{a}\ast}=\text{id}$ if $a> 0$. So we can define
$$ c(u,L):=\text{inf}\,\{\,a\in \mathbb{R} \;|\;i_a^{\phantom{a}\ast}(u)\neq 0\,\}$$
for any $u\neq 0$ in $H^{\ast}(B)$. It follows from Theorem \ref{uniqueness} that $c(u,L)$ is well-defined, i.e. it does not depend on the choice of the generating function used to calculate it. Note also that $c(u,L)$ is a critical value of $S$. The other relevant properties of $c(u,L)$ are contained the following lemma.

\begin{lemma}\label{inv_lagr}
Let $\mu\in H^n(B)$ denote the orientation class of $B$. The map $H^{\ast}(B)\times \mathcal{L}_P\longrightarrow\mathbb{R}$, $(u,L)\longmapsto c(u,L)$ satisfies the following properties:
\vspace{-0.2cm}
\begin{enumerate}
\renewcommand{\labelenumi}{(\roman{enumi})}
\item If $L_1$, $L_2$ have generating functions $S_1$, $S_2: E \longrightarrow\mathbb{R}$ with $|S_1-S_2|_{\mathcal{C}^0}\leq \varepsilon$, then for any $u$ in $H^{\ast}(B)$ it holds that $|c(u,L_1)-c(u,L_2)|\leq \varepsilon$.
\item $$c\big(u\cup v, L_1+L_2\big)\geq c(u,L_1)+c(v,L_2) $$
where $L_1+L_2$ is defined by
$$ L_1+L_2:=\{\;(q,p)\in T^{\ast}B \;|\; p=p_1+p_2, \; (q,p_1)\in L_1, \; (q,p_2)\in L_2  \;\}.$$
\item $$c(\mu,\bar{L})=-c(1,L),$$ where $\bar{L}$ denotes the image of $L$ under the map $T^{\ast}B\rightarrow T^{\ast}B$, $(q,p)\mapsto(q,-p)$.
\item $c(\mu, L)=c(1,L)$ if and only if $L$ is the $0$-section. In this case we have $$c(\mu,L)=c(1,L)=0.$$
\item For any Hamiltonian symplectomorphism $\Psi$ of $T^{\ast}B$ such that $\Psi(P)=P$, it holds 
$$c\big(u,\Psi(L)\big)=c\big(u,L-\Psi^{-1}(0_B)\big).$$
\end{enumerate}
\end{lemma}

The first property is immediate. For $a\in\mathbb{R}$ and $j=1,2$ denote by $\big(E^a\big)_j$ the sublevel set of $S_j$ at $a$, and by $(i_a^{\phantom{a}\ast})_j$ the map on cohomology induced by the inclusion of the pair $\big((E^a)_j\,,\,E^{-\infty}\big)$ into $\big(E\,,\,E^{-\infty}\big)$. If $|S_1-S_2|_{\mathcal{C}^0}\leq \varepsilon$, then we have inclusions of sublevel sets $\big(E^{a-\varepsilon}\big)_2 \subset \big(E^a\big)_1 \subset \big(E^{a+\varepsilon}\big)_2$. For any $a > c(u,L_1)$ we have $(i_a^{\phantom{a}\ast})_1(u)\neq 0$ which implies $(i_{a+\varepsilon}^{\phantom{a+\varepsilon}\ast})_2(u)\neq 0$ and so $c(u,L_2)\leq a +\varepsilon$. Similarly, for any $a'<c(u,L_1)$ we have that $c(u,L_2)>a'-\varepsilon$. It follows that 
$c(u,L_1)-\varepsilon \leq c(u,L_2) \leq c(u,L_1)+\varepsilon$ as we wanted.\\
\\
Properties (ii), (iii) and (iv) require more elaborated arguments of algebraic topology, and we refer to \cite{V} for a proof \footnote{\hspace{1mm}See also \cite{M} for an alternative definition and proof of the main properties of the invariants $c(u,L)$, based on Morse homology.}. We will present here only the proof of (v), because it is the only point that needs arguments of symplectic geometry. We will see in \ref{Bhupalorder} that the analogue statement is not true in the contact case.\\
\\
We first need to introduce some preliminaries from \cite{V} and \cite{V2}. Given Lagrangian submanifolds $L_1$, $L_2$ of $T^{\ast}B$ and points $x$, $y$ in $L_1\cap L_2$, define $$l\,(x,y;L_1,L_2):=\int_{\gamma_1\gamma_2^{-1}}\lambda_{\text{can}}$$ where $\gamma_1$ and $\gamma_2$ are paths in $L_1$, $L_2$ respectively joining $x$ and $y$. Note that $l(x,y;L_1,L_2)=S_1\big(i_{S_1}^{\phantom{S_1}-1}(y)\big)-S_1\big(i_{S_1}^{\phantom{S_1}-1}(x)\big)+S_2\big(i_{S_2}^{\phantom{S_2}-1}(y)\big)-S_2\big(i_{S_2}^{\phantom{S_2}-1}(x)\big)$, where $S_1$, $S_2$ are g.f.q.i. for $L_1$, $L_2$. In particular, for any $L$ in $\mathcal{L}_P$ and $u$ in $H^{\ast}(B)$ there exist points $x$, $y$ in $L \cap 0_B$ such that $c(u,L)=l\,(x,y,;L,0_B)$: just take $x=P$ and $y$ such that $S\,\big(i_S^{\phantom{S_1}-1}(y)\big)=c(u,L)$, where $S$ is a g.f.q.i. for $L$. 
Note that if $\Psi_t$ is an Hamiltonian isotopy of $T^{\ast}B$ then
$l\,(x,y;L_1,L_2)=l\,\big(\Psi_t(x),\Psi_t(y);\Psi_t(L_1),\Psi_t(L_2)\big)$, as can be easily checked using the fact that 
$\Psi_t^{\phantom{t}\ast}\lambda_{\text{can}}-\lambda_{\text{can}}$ is exact. For $L\in\mathcal{L}_P$, define a subset $\Lambda(L)$ of $\mathbb{R}$ by
$\Lambda(L):= \{\,l(x,y,;L,0_B)\,|\, x,y \in L\cap 0_B\,\}$.

\begin{proof}[Proof of Lemma \ref{inv_lagr}(v)] 
Let $\Psi$ be the time-1 flow of a Hamiltonian isotopy $\Psi_t$, and consider the map $t \longmapsto c\big(u,\Psi_t^{\phantom{t}-1}\Psi(L)-\Psi_t^{\phantom{t}-1}(0_B)\big)$. We know by Lemma \ref{inv_lagr}(i) and Theorem \ref{existence} that this map is continuous, and we claim that it takes values in $\Lambda(L)$. Since $\Lambda(L)$ is a totally disconnected set, it will follow that $t \longmapsto c\big(u,\Psi_t^{-1}\Psi(L)-\Psi_t^{-1}(0_B)\big)$ is independent of $t$ and thus in particular $c\big(u,\Psi(L)\big)=c\big(u,L-\Psi^{-1}(0_B)\big)$. To prove the claim, let $x_t$, $y_t$ be points in the intersection of $\Psi_t^{-1}\Psi(L)-\Psi_t^{-1}(0_B)$ with $0_B$ such that
$$ c(u,\Psi_t^{-1}\Psi(L)-\Psi_t^{-1}(0_B))=l(x_t,y_t;\Psi_t^{-1}\Psi(L)-\Psi_t^{-1}(0_B),0_B),$$
and let $x'_t$, $y'_t$ be the corresponding points in $\Psi_t^{-1}\Psi(L)\cap\Psi_t^{-1}(0_B)$. Then we have 
$$c\,\big(u,\Psi_t^{-1}\Psi(L)-\Psi_t^{-1}(0_B)\big)
=l(x_t,y_t;\Psi_t^{-1}\Psi(L)-\Psi_t^{-1}(0_B),0_B)
=l(x'_t,y'_t;\Psi_t^{-1}\Psi(L),\Psi_t^{-1}(0_B))$$
$$=l(\Psi_tx'_t,\Psi_ty'_t;\Psi(L),0_B)\in \Lambda(L)$$
as we wanted.
\end{proof}

\subsection{Invariants for Hamiltonian symplectomorphisms of $\mathbb{R}^{2n}$}\label{inv_symplectomorphisms}

We will now apply the construction of \ref{Vit_inv} to the special case of a compactly supported Hamiltonian symplectomorphism $\phi$ of $\mathbb{R}^{2n}$. We define
$$c(\phi):=c(\mu,\Gamma_{\phi})$$
where $\Gamma_{\phi}$ is the Lagrangian submanifold of $T^{\ast}S^{2n}$ constructed in \ref{symplectomorphisms} and $\mu$ the orientation class of $S^{2n}$. Note that $\Gamma_{\phi}$ intersects the 0-section at the point at infinity of $S^{2n}$. This point plays the role of the point $P$ in \ref{Vit_inv}. We know that $c(\phi)$ is a critical value for any g.f.q.i. of $\Gamma_{\phi}$, and hence that $c(\phi)=\mathcal{A}_{\phi}(q)$ for some fixed point $q$ of $\phi$. Note also that $c(\text{id})=0$. Moreover we have the following properties.

\begin{prop}\label{inv_sympl}
For all $\phi$, $\psi$ in $\text{Ham}^c\,(\mathbb{R}^{2n})$ it holds:
\vspace{-0.2cm}
\begin{enumerate}
\renewcommand{\labelenumi}{(\roman{enumi})}
\item $c(\phi)\geq 0$.
\item If $c(\phi)=c(\phi^{-1})= 0$ then $\phi$ is the identity.
\item $c(\phi\psi)\leq c(\phi) + c(\psi)$.
\item $c(\phi)=c(\psi\phi\psi^{-1})$.
\item If $\phi_1 \leq \phi_2$ in the sense of \ref{symplectomorphisms}, then $c(\phi_1)\leq c(\phi_2)$.
\end{enumerate}
\end{prop}

\begin{proof}
\begin{enumerate}
\renewcommand{\labelenumi}{(\roman{enumi})}
\item We will prove that $c(1,\overline{\Gamma_{\phi}})\leq 0$ for any $\phi$, and then use Lemma \ref{inv_lagr}(iii) to conclude that
$$ c(\phi)=c(\mu,\Gamma_{\phi})=-c(1,\overline{\Gamma_{\phi}}) \geq 0.$$
Since $c(1,\Gamma_{\phi})=\text{inf}\,\{\,a\in \mathbb{R} \;|\;i_{a}^{\phantom{a}\ast}(1)\neq 0\,\}$, we need to prove that
$i_0^{\phantom{0}\ast}(1)\neq 0$. Let $S: E\rightarrow\mathbb{R}$ be a g.f.q.i. for $\overline{\Gamma_{\phi}}$, and recall that we regard $S^{2n}$ as the 1-point compactification $\mathbb{R}^{2n}\cup \{P\}$. Consider the commutative diagram
\begin{displaymath}
\xymatrix{
 H^{\ast}(E^0,E^{-\infty})   \ar[r] &
 H^{\ast}(E_P^{\phantom{P}0},E_P^{\phantom{P}-\infty}) \\
 H^{\ast}(S^{2n}) \ar[r] \ar[u]_{(i_0)^{\ast}} &
 H^{\ast}(\{P\}) \ar[u]_{\cong}}
\end{displaymath}
where the horizontal maps are induced by the inclusions $\{P\}\hookrightarrow S^{2n}$ and $E_P\hookrightarrow E$. Since $\phi$ is compactly supported, $\Gamma_{\phi}$ and hence $\overline{\Gamma_{\phi}}$ coincide with the 0-section on a neighborhood of $P$, so $S_{|E_P}: E_P \rightarrow \mathbb{R}$ is a quadratic form. It follows that the vertical map on the right hand side is an isomorphism. Since the horizontal map on the bottom sends 1 to 1, we see that $i_0^{\phantom{0}\ast}(1)\neq 0$ as we wanted.

\item Note first that $c(u,\Gamma_{\phi})=c(u,\overline{\Gamma_{\phi^{-1}}})$ for all $u$ (apply Lemma \ref{inv_lagr}(v) to $L=0_B$ and $\Psi=\Psi_{\phi}$). Using this, the result then follows from Lemma \ref{inv_lagr}(iii)-(iv).

\item Using (ii),(v) and (iii) of Proposition \ref{inv_lagr} we have
$$ c(\psi)=c(\mu,\Gamma_{\psi})=c\big(\mu\cup 1, \Psi_{\phi^{-1}}(\Gamma_{\phi\psi})\big) 
=c\big(\mu\cup 1, \Gamma_{\phi\psi}-\Psi_{\phi}(0_B)\big)\geq$$
$$ c(\mu,\Gamma_{\phi\psi})+c\big(1,\overline{\Psi_{\phi}(0_B)}\big)
=c(\mu,\Gamma_{\phi\psi})+c(1,\overline{\Gamma_{\phi}})
=c(\mu,\Gamma_{\phi\psi})-c(1,\Gamma_{\phi})$$
$$=c(\phi\psi)-c(\phi)$$
i.e. $c(\phi\psi)\leq c(\phi)+c(\psi)$ as we wanted.

\item Let $\psi$ be the time-1 map of a Hamiltonian isotopy $\psi_t$, and consider the map $t \mapsto c(\psi_t\phi\psi_t^{\phantom{t}-1})$. We know that this map is continuous (by Lemma \ref{inv_lagr}(i) and Theorem \ref{existence}) and that it takes values in the totally disconnected set $\Lambda(\phi)$, since $\Lambda(\psi_t\phi\psi_t^{\phantom{t}-1})=\Lambda(\phi)$ (see for incstance \cite[5.2]{HZ}). It follows that it is independent of $t$, so in particular $c(\phi)=c(\psi\phi\psi^{\phantom{t}-1})$.

\item We know by Proposition \ref{parordsymp} that there are generating functions $S_{\phi_1}$, $S_{\phi_2}$ for $\Gamma_{\phi_1}$, $\Gamma_{\phi_2}$ respectively such that $S_{\phi_1}\leq S_{\phi_2}$. So for any $a$ we have inclusion of sublevel sets $(E^a)_{S_{\phi_2}}\subset (E^a)_{S_{\phi_1}}$ and this easily implies that $c(u,\Gamma_{\phi_1})\leq c(u,\Gamma_{\phi_2})$ for any $u$. In particular, $c(\phi_1)\leq c(\phi_2)$ as we wanted.
\end{enumerate}
\end{proof}

\subsection{The Viterbo partial order}\label{Vitorder}

The Viterbo partial order $\leq_V$ on $\text{Ham}^c\,(\mathbb{R}^{2n})$ is defined as follows. Given $\phi_1$, $\phi_2$ in $\text{Ham}^c\,(\mathbb{R}^{2n})$ we set
$$ \phi_1 \leq_V \phi_2 \quad \text{if} \quad c(\phi_1\phi_2^{\phantom{2}-1})=0.$$
Using the properties in Proposition \ref{inv_sympl} it is immediate to see that $\leq_V$ is indeed a partial order, that it is bi-invariant (i.e. if $\phi_1 \leq_V \phi_2$ and $\psi_1 \leq_V \psi_2$ then $\phi_1\psi_1 \leq_V \phi_2\psi_2$), and that if $\phi_1 \leq \phi_2$ in the sense of \ref{symplectomorphisms} then $\phi_1 \leq_V \phi_2$. In particular this implies that $\leq$ is also a partial order.

\subsection{The Viterbo capacity}\label{Vit_capacity}

Given an open and bounded domain $\mathcal{U}$ of $\mathbb{R}^{2n}$, its \textbf{Viterbo capacity} is defined by
$c(\mathcal{U}):=\text{sup}\,\{\,c(\phi) \;|\; \phi\in\text{Ham}\,(\mathcal{U})\,\}$ where $\text{Ham}\,(\mathcal{U})$ denotes the set of time-1 maps of Hamiltonian functions supported in $\mathcal{U}$. By the following lemma, $c(\mathcal{U})$ is a finite real number.

\begin{lemma}\label{questo5}
If $\phi\in\text{Ham}\,(\mathcal{U})$  and $\psi$ is such that $\psi(\mathcal{U})\cap\mathcal{U}= \emptyset$, then $c(\phi)\leq\gamma(\psi)$ where $\gamma(\psi):=c(\psi)+c(\psi^{-1})$.
\end{lemma}
\begin{proof} 
We first show that under the hypotheses of the lemma we have $c(\psi\phi)=c(\psi)$. Let $x_t$ be a fixed point for $\psi\phi_t$ such that $c(\psi\phi_t)=\mathcal{A}_{\psi\phi_t}(x_t)$. Since $\psi(\mathcal{U})\cap\mathcal{U}= \emptyset$, we see that $x_t\notin\mathcal{U}$. It follows that $x_t$ is also a fixed point for all $\phi_t$, hence for $\psi$. Moreover
$\mathcal{A}_{\psi\phi_t}(x_t)=\mathcal{A}_{\psi}(x_t)$. Thus the continuous map $t \mapsto c(\psi\phi_t)$ takes values in $\Lambda(\psi)$ and hence is independent of $t$. In particular we get that $c(\psi\phi)=c(\psi)$ as we claimed. Using this and Proposition \ref{inv_sympl}(iii) it then follows that
$$c(\phi)\leq c(\psi\phi)+c(\psi^{-1})=c(\psi)+c(\psi^{-1})=\gamma(\psi).$$
\end{proof}

We can extend the definition to arbitrary domains of $\mathbb{R}^{2n}$ by setting
$$ c(\mathcal{V}):=\text{sup}\;\{\;c(\mathcal{U}) \;|\; \mathcal{U}\subset\mathcal{V}, \;\mathcal{U} \;\text{bounded} \;\}$$
if $\mathcal{V}$ is open, and
$$ c(A):=\text{inf}\;\{\;c(\mathcal{V}) \;|\; \mathcal{V} \;\text{open,}\;A\subset\mathcal{V}\;\}$$
for an arbitrary domain $A$.

\begin{thm}
$c$ is a (relative) capacity in $\mathbb{R}^{2n}$, i.e. it satisfies the following properties:
\vspace{-0.2cm}
\begin{enumerate}
\renewcommand{\labelenumi}{(\roman{enumi})}
\item (Symplectic Invariance) For any Hamiltonian symplectomorphism $\psi$ of $\mathbb{R}^{2n}$ we have $$c(\psi(\mathcal{U}))=c(\mathcal{U}).$$
\item (Monotonicity) If $\mathcal{U}_1\subset\mathcal{U}_2$, then $c(\mathcal{U}_1)\leq c(\mathcal{U}_2)$.
\item (Conformality) $c(\alpha\mathcal{U})=\alpha^2c(\mathcal{U})$ for any positive constant $\alpha$.
\item (Non-triviality) $c\,(B^{2n}(1))>0$ and $c\,(C^{2n}(1))<\infty$.
\end{enumerate}
\end{thm}
\begin{proof}
If $\phi\in\text{Ham}\,(\mathcal{U})$ then $\psi\phi\psi^{-1}\in\text{Ham}\,\big(\psi(\mathcal{U})\big)$, thus symplectic invariance follows from Proposition \ref{inv_sympl}(iv). Monotonicity is immediate from the definition, and non-triviality will be discussed in the example below. As for conformality, it can be seen as follows. Consider first a conformal symplectomorphism $\psi$ of $\mathbb{R}^{2n}$, i.e. $\psi^{\ast}\omega=\alpha\,\omega$ for some constant $\alpha$. Then $\Lambda\,(\psi\phi\psi^{-1})=\alpha\,\Lambda(\phi)$ (see \cite{HZ}, 5.2). Suppose that $\psi$ is isotopic to the identity through conformal symplectomorphisms, i.e. $\psi=\psi_t|_{t=1}$ with $\psi_t^{\,\ast}\omega=\alpha(t)\,\omega$ for some function $\alpha(t)$ with $\alpha(0)=1$ and $\alpha(1)=\alpha$. The continuous map $t\mapsto\frac{1}{\alpha(t)}\,c\,(\psi_t\phi\psi_t^{\,-1})$ takes values in the totally disconnected set $\Lambda(\phi)$, thus it is independent of $t$ and so in particular $c\,(\psi\phi\psi^{-1})=\alpha\,c(\phi)$. Applying this to the conformal symplectomorphism $\psi$: $(x,y)\mapsto (\alpha x, \alpha y)$ we get $c\,(\psi\phi\psi^{-1})=\alpha^2\,c(\phi)$. Since $\psi\phi\psi^{-1}\in\text{Ham}\,(\alpha\mathcal{U})$ if $\phi\in\text{Ham}\,(\mathcal{U})$, it follows that $c(\alpha\mathcal{U})=\alpha^2c(\mathcal{U})$ as we wanted.
\end{proof}

\begin{ex}\label{cap_ell}
Consider the ellipsoid
$$E(\alpha_1,\cdots,\alpha_n):=\{\,\frac{1}{\alpha_1}|z_1|^2+\cdots+\frac{1}{\alpha_n}|z_n|^2<1\,\}\subset\mathbb{R}^{2n}\equiv\mathbb{C}^n$$
where $0<\alpha_1\leq\alpha_2\leq\cdots\leq\alpha_n<\infty$. Using Traynor's calculations of symplectic homology of $E(\alpha_1,\cdots,\alpha_n)$ it is easy to see that $c\,\big(E(\alpha_1,\cdots,\alpha_n)\big)=\pi\alpha_1$ (see also \cite{Her}), in particular $c\big(B(R)\big)=R$. Since any bounded domain contained in $C^{2n}(R)$ is also contained in some ellipsoid $E(\alpha_1,\cdots,\alpha_n)$ with $\alpha_1=R$, it follows by monotonicity that $c\,(C^{2n}(R))=R$.
\end{ex}

\subsection{Symplectic homology} \label{sympl_hom}

We will now associate homology groups first to a compactly supported Hamiltonian symplectomorphism of $\mathbb{R}^{2n}$, by considering relative homology of sublevel sets of its generating function, and then, by a limit process, to domains of $\mathbb{R}^{2n}$. In this section we follow \cite{T} although we give a different proof of symplectic invariance of the homology groups (Proposition \ref{conjsympl}).\\
\\
Let $\phi$ be a compactly supported Hamiltonian symplectomorphism of $\mathbb{R}^{2n}$. Given real numbers $a$, $b$ not belonging to the action spectrum of $\phi$ and such that $-\infty<a<b\leq\infty$, we define the \textbf{$k$-th symplectic homology group} of $\phi$ with respect to the values $a,b$ by
$$G_k^{\;\;(a,b]}\,(\phi):=H_{k+\iota}\,(E^b, E^a) $$
where $E^c$, for $c\in \mathbb{R}\cup\infty$, denotes the sublevel set $\{\,x\in E\,|\,S(x)\leq c\,\}$ of a generating function $S:E\rightarrow \mathbb{R}$ for $\phi$ and $\iota$ is the index of the quadratic at infinity part of $S$. It follows from Theorem \ref{uniqueness} that the $G_k^{\;\;(a,b]}\,(\phi)$ are well-defined, i.e. do not depend on the choice of the generating function (see also \cite[3.6]{T}). Moreover, we will prove now that they are invariant by conjugation with a Hamiltonian symplectomorphism.

\begin{prop}\label{conjsympl}
For any $\phi$ and $\psi$ in $\text{Ham}^c\,(\mathbb{R}^{2n})$ we have an induced isomorphism 
$$\psi_{\ast}: G_{\ast}^{\;\;(a,b]}\,(\psi\phi\psi^{-1})\longrightarrow G_{\ast}^{\;\;(a,b]}\,(\phi).$$
\end{prop}

\noindent
To prove this we will need the following lemma.
\begin{lemma}\label{pseudomorse}
Let $f_t$, $t\in [0,1]$, be a continuous 1-parameter family of functions defined on a compact manifold $M$. Suppose that $a\in \mathbb{R}$ is a regular value of all $f_t$. Then there exists an isotopy $\theta_t$ of $M$ such that $\theta_t(M^a_{\phantom{a}0})=M^a_{\phantom{a}t}$, where $M^a_{\phantom{a}t}:=\{\, x\in M \, |\, f_t(x)\leq a\,\}$.
\end{lemma}
\begin{proof}
Since $a$ is a regular value of $f_t$ for all $t\in [0,1]$, there exists an $\varepsilon>0$ such that there are no critical values of any $f_t$ in the interval $(a-\varepsilon,a+\varepsilon )$. Take a $\delta>0$ such that if $|t-s|<\delta$ then $|f_t(x)-f_s(x)|<\varepsilon$ for all $x\in M$, and consider a sequence $0=t_0<t_1<\cdots <t_{k-1}<t_k=1$ with $|t_i-t_{i-1}|<\delta$ for all $i=1,\cdots,k$. For $t_{i-1}<t<t_i$ define a diffeomorphism $\theta_t^{\phantom{t}i}:f_{t_{i-1}}^{\phantom{t_{i-1}}-1}(a)\rightarrow f_t^{\phantom{t}-1}(a)$ by sending a point $x$ of $f_{t_{i-1}}^{\phantom{t_{i-1}}-1}(a)$ to the point obtained by following the flow of the (normalized) gradient $\bigtriangledown f_t$ for a time $a-f_t(x)$. Note that by construction $\bigtriangledown f_t$ will never be $0$ in this process. Note also that (after taking a smaller subdivision if needed) $\bigtriangledown f_t$ is transverse to $f_{t_{i-1}}^{\phantom{t_{i-1}}-1}(a)$, so
$\theta_t^{\phantom{t}i}$ is indeed a diffeomorphism. We can now define a 1-parameter family of diffeomorphisms $\theta_t:f_0^{\phantom{0}-1}(a)\rightarrow f_t^{\phantom{t}-1}(a)$ by defining inductively $\theta_t=\theta_t^{\phantom{t}i}\circ \theta_{t_{i-1}}$ for $t_{i-1}<t<t_i$. A global isotopy as in the statement is now obtained by applying the isotopy extension theorem.
\end{proof}

\begin{proof}[Proof of Proposition \ref{conjsympl}]
Let $\psi_t$ be a Hamiltonian isotopy starting at the identity and ending at $\psi_1=\psi$. We have $\Lambda\,\big(\psi_t\phi\psi_t^{-1}\big)=\Lambda(\phi)$ for all $t$ thus if we consider a continuous family $S_t:\mathbb{R}^{2n}\times\mathbb{R}^{N}\longrightarrow \mathbb{R}$ of generating functions, each $S_t$ generating the corresponding $\psi_t\phi\psi_t^{-1}$, then by Lemma \ref{crucialsympl} the set $\Lambda\,\big(\psi_t\phi\psi_t^{-1}\big)$ of critical values of $S_t$ is independent of $t$. Since $a$ and $b$ are regular values for $S_0$ it follows that they are regular values for all $S_t$, and so we can conclude using an analogue of Lemma \ref{pseudomorse} for pairs of sublevel sets. Note that we can do it even though $\mathbb{R}^{2n}\times\mathbb{R}^{N}$ is not compact, because the functions $S_t$ are (special) quadratic at infinity.
\end{proof}

\noindent
Consider now a domain $\mathcal{U}$ of $\mathbb{R}^{2n}$. Given $a,b\in\mathbb{R}$ we denote by $\text{Ham}_{a,b}^{\;\;\;\;\;c}\,(\mathcal{U})$ the set of compactly supported Hamiltonian symplectomorphisms of $\mathbb{R}^{2n}$ that are the time-1 map of a Hamiltonian function which is supported in $\mathcal{U}$ and whose action spectrum does not contain $a$ and $b$. Note that $\text{Ham}_{a,b}^{\;\;\;\;\;c}\,(\mathcal{U})$ is directed with respect to the partial order $\leq$, i.e. for any $\phi$, $\psi$ in $\text{Ham}_{a,b}^{\;\;\;\;\;c}\,(\mathcal{U})$ there is a $\varphi$ in $\text{Ham}_{a,b}^{\;\;\;\;\;c}\,(\mathcal{U})$ such that $\phi\leq\varphi$ and $\psi\leq\varphi$. Recall that if $\phi_1\leq\phi_2$ we have an induced homomorphism $\lambda_1^{\phantom{1}2}: G_k^{\;\;(a,b]}\,(\phi_2)\longrightarrow G_k^{\;\;(a,b]}\,(\phi_1)$. Note that given $\phi_1$, $\phi_2$, $\phi_3$ in $\text{Ham}_{a,b}^{\;\;\;\;\;c}\,(\mathcal{U})$ with $\phi_1\leq\phi_2\leq\phi_3$, it holds $\lambda_3^{\phantom{3}2}\circ \lambda_2^{\phantom{2}1}= \lambda_3^{\phantom{3}1}$ and $\lambda_i^{\phantom{i}i}=\text{id}$. This means in particular that $\lbrace G_k^{\;\;(a,b]}\,(\phi_i)\rbrace_{\phi_i\in\text{Ham}_{a,b}^{\;\;\;\;\;c}\,(\mathcal{U})}$ is an inversely directed family of groups, so we can define the 
\textbf{$k$-th symplectic homology group} $G_k^{\;\;(a,b]}\,(\mathcal{U})$ of $\mathcal{U}$ with respect to the values $a,b$ to be the inverse limit of this family. Note that $G_k^{\;\;(a,b]}\,(\mathcal{U})$ can be calculated by any sequence $\phi_1\leq\phi_2\leq\phi_3\leq \cdots$ such that the associated Hamiltonians get arbitrarily large.

\begin{thm}[Symplectic Invariance]
For any domain $\mathcal{U}$ in $\mathbb{R}^{2n}$ and any Hamiltonian symplectomorphism $\psi$ we have an induced isomorphism $\psi_{\ast}:G_{\ast}^{\;\;(a,b]}\,\big(\psi(\mathcal{U})\big)\longrightarrow G_{\ast}^{\;\;(a,b]}\,(\mathcal{U})$.
\end{thm}
\begin{proof}
Let $\phi_1\leq\phi_2\leq\phi_3\leq \cdots$ be an unbounded ordered sequence supported in $\mathcal{U}$. Then $\psi\phi_1\psi^{-1}\leq\psi\phi_2\psi^{-1}\leq\psi\phi_3\psi^{-1}\leq \cdots$  is an unbounded ordered sequence supported in $\psi(\mathcal{U})$. By Proposition \ref{conjsympl} we have isomorphisms $\psi_i^{\:\ast}: G_{\ast}^{\;\;(a,b]}\,(\psi\phi_i\psi^{-1})\longrightarrow G_{\ast}^{\;\;(a,b]}\,(\phi_i)$, commuting with the $\lambda_i^{\phantom{i}j}$ of the limit process. Thus we get an induced isomorphism between $G_{\ast}^{\;\;(a,b]}\,(\psi(\mathcal{U}))$ and $G_{\ast}^{\;\;(a,b]}\,(\mathcal{U})$.
\end{proof}

\begin{thm}[Monotonicity]
Every inclusion of domains induces a homomorphism of homology groups (reversing the order) with the following functorial properties: 
\begin{enumerate}
\renewcommand{\labelenumi}{(\roman{enumi})}
\item
If $\mathcal{U}_1\subset\mathcal{U}_2\subset\mathcal{U}_3$ then the following diagram commutes
\begin{displaymath}
\xymatrix{
 G_{\ast}^{\;\;(a,b]}\,(\mathcal{U}_3) \ar[r] \ar[rd] &
 G_{\ast}^{\;\;(a,b]}\,(\mathcal{U}_2) \ar[d] \\
 & G_{\ast}^{\;\;(a,b]}\,(\mathcal{U}_1).}
\end{displaymath} 
\item
If $\mathcal{U}_1\subset\mathcal{U}_2$, then for any Hamiltonian symplectomorphism $\psi$ the following diagram commutes
\begin{displaymath}
\xymatrix{
 G_{\ast}^{\;\;(a,b]}\,(\mathcal{U}_2) \ar[r] &
 G_{\ast}^{\;\;(a,b]}\,(\mathcal{U}_1)  \\
 G_{\ast}^{\;\;(a,b]}\,\big(\psi(\mathcal{U}_2)\big) \ar[u]^{\psi_{\ast}} \ar[r] &  
 G_{\ast}^{\;\;(a,b]}\,\big(\psi(\mathcal{U}_1)\big) \ar[u]_{\psi_{\ast}}.}
\end{displaymath} 
\end{enumerate}
\end{thm}
\begin{proof}
Suppose $\mathcal{U}_1\subset\mathcal{U}_2$.
Given an unbounded ordered sequence $\phi_1^{\;2}\leq\phi_2^{\;2}\leq\phi_3^{\;2}\leq \cdots$ supported in $\mathcal{U}_2$, there exists an unbounded ordered sequence $\phi_1^{\;1}\leq\phi_2^{\;1}\leq\phi_3^{\;1}\leq \cdots$ supported in $\mathcal{U}_1$ such that $\phi_i^{\;1}\leq\phi_i^{\;2}$. The homomorphisms
$G_{\ast}^{\;\;(a,b]}\,(\phi_i^{\;2})\longrightarrow G_{\ast}^{\;\;(a,b]}\,(\phi_i^{\;1})$ induce a homomorphism of the inverse limits $G_{\ast}^{\;\;(a,b]}\,(\mathcal{U}_1)\rightarrow G_{\ast}^{\;\;(a,b]}\,(\mathcal{U}_2)$. The functorial properties are easy to check.
\end{proof}

Traynor \cite{T} calculated the homology groups with $\mathbb{Z}_2$-coefficients of ellipsoids in $\mathbb{R}^{2n}$. We will need the following special case of her calculations.

\begin{thm}\label{thmballs}
Consider $B(R)\subset\mathbb{R}^{2n}$ and let $a$ be a positive real number. Then for $\ast=2nl$ we have
$$ G_{\ast}^{\;\;(a,\infty]}\,\big(B(R)\big)  = \left\{ 
\begin{array}{l l}
 \mathbb{Z}_2 & \quad \text{if}\;\; \frac{a}{l}<R\leq\frac{a}{l-1} \\
 0            & \quad \text{otherwise}
\end{array} \right. $$
where $l$ is any positive integer. In particular for $l=1$ we have
$$ G_{2n}^{\;\;(a,\infty]}\,\big(B(R)\big)  = \left\{ 
\begin{array}{l l}
 \mathbb{Z}_2 & \quad \text{if}\;\; R>a \\
 0            & \quad \text{otherwise}.
\end{array} \right. $$
For all other values of $\ast$ the corresponding homology groups are zero. Moreover, given $R_1$, $R_2$ with $\frac{a}{l}< R_2 < R_1 \leq \frac{a}{l-1}$, the homomorphism $G_{\ast}^{\;\;(a,\infty]}\,\big(B(R_1)\big)\longrightarrow G_{\ast}^{\;\;(a,\infty]}\,\big(B(R_2)\big)$ induced by the inclusion $B(R_2)\subset B(R_1)$ is an isomorphism.
\end{thm}

\section{Contact Capacity and Homology for Domains in $\mathbb{R}^{2n}\times S^1$}\label{contact}

We refer to \cite{Ge} for an introduction to Contact Topology, and discuss here only some basic preliminaries.\\
\\
A contact manifold is an odd dimensional manifold $V^{2n+1}$ endowed with a hyperplanes field $\xi$ which is maximally non-integrable, i.e. it is locally the kernel of a 1-form $\eta$ such that $\eta\wedge (d\eta)^n$ never vanishes. We will always assume that the contact manifold is cooriented, i.e. that $\eta$ is globally defined.
Standard examples of contact manifolds can be obtained by considering the prequantization space of an exact symplectic manifold $\big(M,\omega=-d\lambda\big)$, i.e. the manifold $M\times \mathbb{R}$ endowed with the contact structure $\xi=\text{ker}\,(dz-\lambda)$ where $z$ is the coordinate on $\mathbb{R}$. Special instances of this construction are the standard contact euclidean space $\big(\mathbb{R}^{2n+1},\xi_0=\text{ker}\,(dz-ydx)\big)$, which is the prequantization of $(\mathbb{R}^{2n}, \omega_0)$, and the 1-jet bundle $J^1B$ of a manifold $B$, which is the prequantization of $(T^{\ast}B, \omega_{\text{can}})$.\\
\\
A diffeomorphism $\phi$ of a contact manifold $\big(V,\,\xi=\text{ker}(\eta)\big)$ is called a contactomorphism if its differential preserves $\xi$ and its coorientation. It is called a strict contactomorphism if $\phi^{\ast}\eta=\eta$. A time-dependent vector field $X_t$ on $V$ is called a contact vector field if its flow consists of contactomorphisms. Given a time-dependent function $H_t$ on $V$ there exists a unique contact vector field $X_t$ such that $\eta(X_t)=H_t$ (see \cite[Section 2.3]{Ge}). The function $H_t$ is then called the contact Hamiltonian of the flow $\phi_t$ of $X_t$, with respect to the contact form $\eta$. An immersion $i: L\rightarrow \big(V,\,\xi=\text{ker}(\eta)\big)$ is called isotropic if $i^{\ast}\eta=0$ and Legendrian if moreover the dimension of $L$ is maximal, i.e. half of $(\text{dim}(M)-1)$. For example, if $V$ is the prequantization of an exact symplectic manifold $\big(M,\omega=-d\lambda\big)$ and $i:L \rightarrow M$ is an exact Lagrangian immersion with $i^{\ast}\lambda=df$, then the lift $i\times f$ is a Legendrian immersion of $L$ into $V=M\times\mathbb{R}$. Note that in particular, up to addition of a constant in the $\mathbb{R}$-coordinate, this gives a 1-1 correspondence between Legendrian submanifolds of $V$ and exact Lagrangian submanifolds of $M$.\\
\\
In the contact case, generating functions are defined for Legendrian submanifolds of $J^1B$. A Lagrangian submanifold of $T^{\ast}B$ that is Hamiltonian isotopic to the 0-section is in particular exact, and we will see that it has the same generating function as its lift to $J^1B$. This basic fact is what is behind the relation between the symplectic invariants defined in the previous section and the contact invariants that we are going to define now.

\subsection{Generating functions for Legendrian submanifolds of $J^1B$}\label{legendrian}

Consider a real function $f$ defined on a smooth manifold $B$. The 1-jet of $f$ is the Legendrian immersion $j^1f:B\rightarrow J^1B$ defined by $x\mapsto \big(x,df(x),f(x)\big)$. Note that $j^1f$ is the lift of the differential of $f$, seen as an exact Lagrangian immersion $B\rightarrow T^{\ast}B$.
More generally, given a transverse variational family $(E,S)$ over $B$ denote by $j_S: \Sigma_S \rightarrow J^1B$ the lift of the exact Lagrangian immersion $i_S:\Sigma_S\longrightarrow T^{\ast}B$ defined in \ref{lagrsubmfds}, i.e. $j_S(e)=\big(p(e),v^{\ast}(e), S(e)\big)$. Then $S:E\longrightarrow\mathbb{R}$ is called a \textbf{generating function} for the Legendrian submanifold $\widetilde{L_S}:=j_S\,(\Sigma_S)$ of $J^1B$. Note that critical points of $S$ correspond under $j_S$ with intersection points of $\widetilde{L_S}$ with the 0-wall of $J^1B$ (which is defined to be the product of the 0-section of $T^{\ast}B$ with $\mathbb{R}$), and that the corresponding critical value is the $\mathbb{R}$-coordinate of the intersection point with the 0-wall. Moreover, non-degenerate critical points correspond to transverse intersections (see \cite[Proposition 2.1]{C}). Note also that if two functions differ by an additive constant, then they generate different Legendrian submanifolds of $J^1B$ (in fact different lifts of the same Lagrangian submanifold of $T^{\ast}B$).\\
\\
The existence and uniqueness theorems for generating functions have been generalized to the contact case by Chaperon, Chekanov and Th\'{e}ret.

\begin{thm}[\cite{Ch}, \cite{C}, \cite{Th}]\label{eugfcont}
If $B$ is closed, then any Legendrian submanifold of $J^1B$ contact isotopic to the 0-section has a g.f.q.i., which is unique up to fiber-preserving diffeomorphism and stabilization. If $L\subset J^1B$ has a g.f.q.i. and $\psi_t$ is a contact isotopy of $J^1B$, then there exists a continuous family of g.f.q.i. $S_t:E\longrightarrow\mathbb{R}$ such that each $S_t$ generates the corresponding $\psi_t(L)$.
\end{thm}

As in the symplectic case, any g.f.q.i. is equivalent to a special one. We will always assume generating functions to be special whenever this is needed.

\subsection{Generating functions for contactomorphisms of $\mathbb{R}^{2n+1}$}\label{contactomorphisms}

In order to apply the results of the previous section to contactomorphisms of $\mathbb{R}^{2n+1}$ we need to associate to a contactomorphism of $\mathbb{R}^{2n+1}$ a Legendrian submanifold in some 1-jet bundle. Moreover, we should do this in a way which is compatible with the construction given in the symplectic case. By this we mean the following. Recall that any Hamiltonian symplectomorphism $\varphi$ of $\mathbb{R}^{2n}$ can be lifted to a contactomorphism $\widetilde{\varphi}$ of $\mathbb{R}^{2n+1}$. To get a simple relation between the contact invariants that we will define in this section and the symplectic ones defined before, we need the generating function of $\widetilde{\varphi}$ to be essentially the same as the generating function of $\varphi$. We now explain how this can be done, following Bhupal \cite{B}. Let $\phi$ be a contactomorphism of $\mathbb{R}^{2n+1}$, with $\phi^{\ast}(dz-ydx)=e^g(dz-ydx)$ for some function $g:\mathbb{R}^{2n+1}\longrightarrow\mathbb{R}$. Consider the graph of $\phi$, i.e. the embedding
$$
\text{gr}_{\phi}:\mathbb{R}^{2n+1}\longrightarrow \mathbb{R}^{2(2n+1)+1}\;,\;\; q\mapsto (q,\phi(q),g(q)).
$$
If we endow $\mathbb{R}^{2(2n+1)+1}$ with the contact structure given by the kernel of $e^{\theta}\,(dz-ydx)-(dZ-YdX)$, then $\text{gr}_{\phi}$ becomes a Legendrian embedding. Define now $\Gamma_{\phi}:\mathbb{R}^{2n+1}\longrightarrow J^1\mathbb{R}^{2n+1}$
to be the composition $\Gamma_{\phi}=\tau\circ\text{gr}_{\phi}$, where $\tau:\mathbb{R}^{2(2n+1)+1}\longrightarrow J^1\mathbb{R}^{2n+1}$ is the contact embedding defined by $$(x,y,z,X,Y,Z,\theta)\mapsto \big(x,Y,z, Y-e^{\theta}y, x-X, e^{\theta}-1, xY-XY+Z-z\big).$$
Thus
\begin{equation}\label{cazzuta}
\Gamma_{\phi}(x,y,z)=\big(x,\phi_2,z,\phi_2-e^gy, x-\phi_1, e^g-1, x\phi_2-\phi_1\phi_2+\phi_3-z\big).
\end{equation}
To motivate this formula, consider the case of the lift $\widetilde{\varphi}$ of a Hamiltonian symplectomorphism $\varphi$ of $\mathbb{R}^{2n}$. Recall that $\widetilde{\varphi}$ is defined by $\widetilde{\varphi}(x,y,z)=\big(\varphi_1(x,y),\varphi_2(x,y),z+F(x,y)\big)$ where $F$ is the compactly supported function satisfying $\varphi^{\ast}\lambda_0-\lambda_0=dF$. In \ref{symplectomorphisms} we associated to $\varphi$ the Lagrangian embedding $\Gamma_{\varphi}:\mathbb{R}^{2n}\longrightarrow T^{\ast}\mathbb{R}^{2n}\;,\;(x,y)\mapsto\big(x,\varphi_2,\varphi_2-y,x-\varphi_1\big)$. This embedding is exact with $\Gamma_{\varphi}^{\phantom{\varphi}\ast}\lambda_{\text{can}}=d(x\varphi_2-\varphi_1\varphi_2+F)$, thus it can be lifted to the Legendrian embedding $\widetilde{\Gamma_{\varphi}}:\mathbb{R}^{2n}\longrightarrow J^1\mathbb{R}^{2n}\;,\;(x,y)\mapsto\big(x,\varphi_2,\varphi_2-y,x-\varphi_1,x\varphi_2-\varphi_1\varphi_2+F\big)$. Identify now $J^1\mathbb{R}^{2n+1}$ with $J^1\mathbb{R}^{2n}\times T^{\ast}\mathbb{R}$ via $(x,y,z,X,Y,Z,\theta)\mapsto \big((x,y,X,Y,\theta)\,,\,(z,Z)\big)$ and consider the Legendrian embedding $\widetilde{\Gamma_{\varphi}}\times\mbox{$0$-section}$,
$\mathbb{R}^{2n+1}\longrightarrow J^1\mathbb{R}^{2n+1}:
(x,y,z)\mapsto\big(x,\varphi_2,z,\varphi_2-y,x-\varphi_1,0,x\varphi_2-\varphi_1\varphi_2+F\big)$.
Note that, since $\varphi$ is a strict contactomorphism, $\widetilde{\Gamma_{\varphi}}\times\mbox{$0$-section}$ coincides with the Legendrian embedding $\Gamma_{\widetilde{\varphi}}:\mathbb{R}^{2n+1}\longrightarrow J^1\mathbb{R}^{2n+1}$ given by (\ref{cazzuta}). Besides shedding some light to the formula (\ref{cazzuta}) the above discussion proves the following lemma.

\begin{lemma}\label{lift}
If $\varphi$ is a compactly supported Hamiltonian symplectomorphism of $\mathbb{R}^{2n}$ with generating function $S: \mathbb{R}^{2n}\times\mathbb{R}^N\longrightarrow\mathbb{R}$, then the function $\widetilde{S}: \mathbb{R}^{2n+1}\times\mathbb{R}^N\longrightarrow\mathbb{R}$ defined by $\widetilde{S}(x,y,z;\xi)=S(x,y;\xi)$ is a generating function for the lift $\widetilde{\varphi}$.
\end{lemma}

Similarly to the symplectic case, for a contactomorphisms $\phi$ of $\mathbb{R}^{2n+1}$ we can write $\Gamma_{\phi}$ also as
$\Gamma_{\phi}=\Psi_{\phi}(\text{0-section})$, with $\Psi_{\phi}$ denoting the local contactomorphism of $J^1\mathbb{R}^{2n+1}$ defined by the diagram
\begin{equation}\label{diagram_c}
\xymatrix{
 \quad \mathbb{R}^{2(2n+1)+1} \quad \ar[r]^{\overline{\phi}} \ar[d]_{\tau} &
 \quad \mathbb{R}^{2(2n+1)+1} \quad \ar[d]^{\tau} \\
 \quad J^1\mathbb{R}^{2n+1} \quad \ar[r]_{\Psi_{\phi}} &  \quad J^1\mathbb{R}^{2n+1}}\quad
\end{equation}
where $\overline{\phi}$ is the contactomorphism $(p,P,\theta)\mapsto (p,\phi(P), g(P)+\theta)$. This shows in particular that if $\phi$ is contact isotopic to the identity then $\Gamma_{\phi}$ is contact isotopic to the 0-section. Suppose indeed that $\phi$ is the time-1 map of a contact isotopy $\phi_t$. Then we get a local contact isotopy $\Psi_{\phi_t}$ of $J^1\mathbb{R}^{2n+1}$ connecting $\Psi_{\phi}$ to the identity. By the contact isotopy extension theorem (see \cite[Section 2.6]{Ge}) we can extend this local isotopy to a global one, so we see that $\Gamma_{\phi}$ is contact isotopic to the 0-section. Notice that, as in the symplectic case, diagram (\ref{diagram_c}) behaves well with respect to composition: for all contactomorphisms $\phi$, $\phi_1$ and $\phi_2$ we have namely that
$\Psi_{\phi_1}\circ\Psi_{\phi_2}=\Psi_{\phi_1\phi_2} $ (in particular $\Gamma_{\phi_1\,\circ\,\phi_2}=\Psi_{\phi_1}\,(\Gamma_{\phi_2})$) and $\Psi_{\phi}^{\;-1}=\Psi_{\phi^{-1}}$.\\
\\
If $\phi$ is compactly supported then the Legendrian embedding $\Gamma_{\phi}:\mathbb{R}^{2n+1}\longrightarrow J^1  \mathbb{R}^{2n+1}$ coincides with the 0-section outside a compact set, so it can be seen as a Legendrian submanifold of $J^1S^{2n+1}$, which is contact isotopic to the 0-section if $\phi$ is contact isotopic to the identity. By Theorem \ref{eugfcont}, it follows that $\Gamma_{\phi}$ has a generating function, which is unique up to fiber-preserving diffeomorphism and stabilization. The same is true if $\phi$ is a contactomorphism of $\mathbb{R}^{2n+1}$ which is 1-periodic in the $z$-coordinate and compactly supported in the $(x,y)$-plane, because then $\Gamma_{\phi}$ can be seen as a Legendrian submanifold of $J^1(S^{2n}\times S^1)$. We will denote by $\text{Cont}_0^{\phantom{0}c}\,(\mathbb{R}^{2n+1})$ the group of compactly supported contactomorphisms of $\mathbb{R}^{2n+1}$ that are isotopic to the identity, and by $\text{Cont}^{\; \; c}_{\text{1-per}} (\mathbb{R}^{2n+1})$ the group of contactomorphisms of $\mathbb{R}^{2n+1}$ that are 1-periodic in the $z$-coordinate, compactly supported in the $(x,y)$-plane and isotopic to the identity through contactomorphisms of this form. Note that $\text{Cont}^{\; \; c}_{\text{1-per}} (\mathbb{R}^{2n+1})$ can be identified with the group $\text{Cont}_0^{\phantom{0}c}\,(\mathbb{R}^{2n}\times S^1)$ of compactly supported contactomorphisms of $\mathbb{R}^{2n}\times S^1$ isotopic to the identity.\\
\\
Recall that in the symplectic case the set of critical values of a generating function coincides with the action spectrum of the generated Hamiltonian symplectomorphism. Before stating the contact analogue of this crucial result we need to introduce the following terminology. Given a contactomorphism $\phi$ of $\mathbb{R}^{2n+1}$ with $\phi^{\ast}(dz-ydx)=e^g(dz-ydx)$, we say that $q=(x,y,z)$ is a \textbf{translated point} for $\phi$ if $\phi_1(q)=x$, $\phi_2(q)=y$ and $g(q)=0$. In analogy to the symplectic case we will call $\phi_3(q)-z$ the \textbf{contact action} of $\phi$ at the translated point $q$. 

\begin{lemma}\label{crucialcont}
Let $\phi$ be a contactomorphism of $\mathbb{R}^{2n+1}$ with generating function $S$. Then a point $q=(x,y,z)$ of $\mathbb{R}^{2n+1}$ is a translated point of $\phi$ if and only if $\big(q,0,\phi_3(q)-z\big)\in\Gamma_{\phi}$, and so if and only if $i_S^{\;-1}\big(q,0,\phi_3(q)-z\big)$ is a critical point of $S$. In this case the corresponding critical value is the contact action $\phi_3(q)-z$.
\end{lemma}
\begin{proof}
If $q$ is a translated point then $\big(q,0,\phi_3(q)-z\big)=\Gamma_{\phi}(q)\in\Gamma_{\phi}$. Conversely, it is easy to see that if $\big(q,0,\phi_3(q)-z\big)=\Gamma_{\phi}(q_0)$ for some $q_0\in\mathbb{R}^{2n+1}$ then $q_0=q$ and $q$ is a translated point. Recall then from \ref{legendrian} that intersections of $\Gamma_{\phi}$ with the 0-wall correspond to critical points of the generating function $S$, with critical value given by the last coordinate.
\end{proof}

\noindent
Consider for example the lift $\widetilde{\varphi}$ of a Hamiltonian symplectomorphism $\varphi$ of $\mathbb{R}^{2n}$. Recall that 
$\widetilde{\varphi}$ is defined by $\widetilde{\varphi}(x,y,z)=\big(\varphi_1(x,y),\varphi_2(x,y),z+F(x,y)\big)$. A point $(x,y,z)$ of $\mathbb{R}^{2n+1}$ is a translated point for $\widetilde{\varphi}$ if and only if $(x,y)$ is a fixed point of $\varphi$, and the contact action is given by $F(x,y)=\mathcal{A}_{\varphi}(x,y)$. Note that this, together with Lemma \ref{hamis}, gives an alternative proof of the fact that the set of critical values of the generating function of $\varphi$ coincides with the action spectrum of $\varphi$.\\
\\
Similarly to the symplectic case we can define a relation $\leq$ on the groups $\text{Cont}_0^{\phantom{0}c}\,(\mathbb{R}^{2n+1})$ and $\text{Cont}^{\; \; c}_{\text{1-per}} (\mathbb{R}^{2n+1})$ by setting $\phi_1\leq\phi_2$ if $\phi_2\phi_1^{-1}$ is the time-1 flow of some non-negative contact Hamiltonian. We will see in \ref{Bhupalorder} that this relation is in fact a partial order. In the rest of this section we will show that the analogue of Proposition \ref{parordsymp} is still true in the contact case. We will only consider compactly supported contactomorphisms, but all arguments go through for elements of $\text{Cont}^{\; \; c}_{\text{1-per}} (\mathbb{R}^{2n+1})$ as well.

\begin{prop}\label{parordcont}
Let $\phi_0$, $\phi_1$ be either in $\text{Cont}_0^{\phantom{0}c}\,(\mathbb{R}^{2n+1})$ or in $\text{Cont}^{\; \; c}_{\text{1-per}} (\mathbb{R}^{2n+1})$. If $\phi_0\leq\phi_1$, then there are generating functions $S_0$, $S_1: E \longrightarrow \mathbb{R}$ for $\Gamma_{\phi_0}$, $\Gamma_{\phi_1}$ respectively such that $S_0\leq S_1$.
\end{prop}

\noindent
Note that, by considering the lift of Hamiltonian symplectomorphisms of $\mathbb{R}^{2n}$, this result contains Proposition \ref{parordsymp} as a special case. To prove Proposition \ref{parordcont} we will use the concept of \textit{Greek generating functions} for contactomorphisms of $J^1\mathbb{R}^m$, which was introduced by Chaperon in \cite{Ch}.\\
\\
Let $\varphi$ be a contactomorphism of $J^1\mathbb{R}^m$, and assume it is $\mathcal{C}^1$-close to the identity \footnote{\hspace{1mm}Chaperon showed in fact how to construct a Greek generating function
$\Phi:J^1\mathbb{R}^m\times\mathbb{R}^N\rightarrow\mathbb{R}$ for any compactly supported contactomorphism of $J^1\mathbb{R}^m$ contact isotopic to the identity, in such a way that the corresponding \textit{Latin generating function} is quadratic at infinity. However we will only need the construction of Greek generating functions for $\mathcal{C}^1$-small contactomorphisms.}. Then the \textbf{Greek generating function} of $\varphi$ is a function  $\Phi:\mathbb{R}^m\times(\mathbb{R}^m)^{\ast}\times\mathbb{R}\rightarrow\mathbb{R}$ defined as follows. For $(p,z)\in(\mathbb{R}^m)^{\ast}\times\mathbb{R}$, consider the function $A_{p,z}:\mathbb{R}^m\rightarrow\mathbb{R}$ given by $A_{p,z}(q)=z+pq$. Note that $j^1A_{p,z}:\mathbb{R}^m\rightarrow J^1\mathbb{R}^m$, for $(p,z)$ varying in 
$(\mathbb{R}^m)^{\ast}\times\mathbb{R}$, form a foliation of $J^1\mathbb{R}^m$. Since $\varphi$ is $\mathcal{C}^1$-close to the identity, $\varphi\,(j^1A_{p,z})$ is still a section of $J^1\mathbb{R}^m$ and thus it is the 1-jet of a function $\Phi_{p,z}:\mathbb{R}^m\rightarrow\mathbb{R}$. The Greek generating function $\Phi$ is then defined by $\Phi(Q,p,z)=\Phi_{p,z}(Q)$. The \textit{Latin generating function} of $\varphi$ is the function $F:\mathbb{R}^m\times(\mathbb{R}^m)^{\ast}\times\mathbb{R}\rightarrow\mathbb{R}$ defined by $F(Q,p,z):=\Phi(Q,p,z)-(z+pQ)$. Note that $F$ is identically $0$ if (and only if) $\varphi$ is the identity. Moreover one can show that $F$ is independent of $z$ if and only if $\varphi$ is the lift of an Hamiltonian symplectomorphism of $T^{\ast}\mathbb{R}^m$, and that in this case it coincides with the function constructed by Traynor in \cite[4.4]{T} (but we are not going to need this fact in the following).\\
\\
For the proof of Proposition \ref{parordcont} we will need the following three lemmas.

\begin{lemma}\label{lui}
Consider a Legendrian submanifold $L$ of $J^1\mathbb{R}^m$ with generating function $S:\mathbb{R}^m\times\mathbb{R}^N\rightarrow\mathbb{R}$, and a compactly supported contact isotopy $\varphi_t$ of $J^1\mathbb{R}^m$ which is $\mathcal{C}^1$-close to the identity and has Greek generating function
$\Phi_t:\mathbb{R}^m\times(\mathbb{R}^m)^{\ast}\times\mathbb{R}\rightarrow\mathbb{R}$. Then the function
$S_t:\mathbb{R}^m\times\big((\mathbb{R}^m)^{\ast}\times\mathbb{R}^m\times\mathbb{R}^N\big)\rightarrow\mathbb{R}$ defined by 
$S_t\,(Q;p,q,\xi):=\Phi_t\,\big(Q,p,S(q;\xi)-pq \big)$ is a generating function for $\varphi_t(L)$.
\end{lemma}

\noindent
This lemma can be obtained as a special case of the composition formula (9) in \cite{Ch} (see also Section III of \cite{Th}).

\begin{lemma}[\cite{Ch}, 2.2]\label{lui2}
Let $\varphi_t$ be a contact isotopy of $J^1\mathbb{R}^m$ with contact Hamiltonian $H_t:J^1\mathbb{R}^m\rightarrow\mathbb{R}$. Assume that $\varphi_t$ is $\mathcal{C}^1$-close to the identity and 
has Greek generating function
$\Phi_t:\mathbb{R}^m\times(\mathbb{R}^m)^{\ast}\times\mathbb{R}\rightarrow\mathbb{R}$. Then given $(q,p,z)$ in $\mathbb{R}^m\times(\mathbb{R}^m)^{\ast}\times\mathbb{R}$ it holds
$$\dfrac{d\Phi_t}{dt}\bigg\arrowvert_{t=t_0}\,(Q_{t_0},p,z)=H_{t_0}\,\big(Q_{t_0},P_{t_0}, Y_{t_0}\big)$$
where $(Q_{t_0},P_{t_0}, Y_{t_0})=\varphi_t\,(q,p,z+pq)$.
\end{lemma}

\noindent
The next lemma is a special case for $t=0$ of Lemma \ref{lui}, and can also be easily verified directly.

\begin{lemma}\label{lui3}
Consider a Legendrian submanifold $L$ of $J^1\mathbb{R}^m$. If $S:\mathbb{R}^m\times\mathbb{R}^N\rightarrow\mathbb{R}$ is a generating function for $L$, then so is the function
$S_0:\mathbb{R}^m\times(\mathbb{R}^m)^{\ast}\times\mathbb{R}\rightarrow\mathbb{R}$ defined by 
$S_0\,(Q;p,q,\xi):=S(q;\xi)+p(Q-q)$.
\end{lemma}

\begin{proof}[Proof of Proposition \ref{parordcont}] 
Let $\phi_1\phi_0^{\phantom{0}-1}$ be the time-1 map of a contact isotopy $\psi_t$ of $\mathbb{R}^{2n+1}$. We will first prove the result assuming that $\psi_t$ is $\mathcal{C}^1$-close to the identity. Consider the contact isotopy $\Psi_{\psi_t}$ of $J^1\mathbb{R}^{2n+1}$: we know that it is $\mathcal{C}^1$-close to the identity and has non-negative Hamiltonian, because so does $\psi_t$ by assumption. Thus by Lemma \ref{lui2} if $\Psi_t:J^1\mathbb{R}^{2n+1}\rightarrow \mathbb{R}$ is a Greek generating function for $\Psi_{\psi_t}$ then $\dfrac{d}{dt}\Psi_t\geq 0$. Take now a generating function $S:\mathbb{R}^{2n+1}\times\mathbb{R}^N\rightarrow\mathbb{R}$ for $\Gamma_{\phi_0}\subset J^1\mathbb{R}^{2n+1}$. Then, by Lemma \ref{lui}, $\Gamma_{\psi_t\phi_0}=\Psi_{\psi_t}(\Gamma_{\phi_0})$ has generating function 
$S_t\,(Q;p,q,\xi):=\Psi_t\,\big(Q,p,S(q;\xi)-pq \big)$. Thus $\dfrac{d}{dt}S_t\geq 0$, in particular $S_1\geq S_0$. Note that $S_1$ is a generating function for $\Gamma_{\phi_1}$, and $S_0$ is a generating function for $\Gamma_{\phi_0}$ related to $S$ as in Lemma \ref{lui3}. For the general case the result follows by repeating this process and applying Lemma \ref{lui3} at every step. This can be done because it can be proved (see Lemma 1 in Section 2.4 of \cite{Ch}) that there exists a $\delta>0$ such that every $\psi_t\psi_s^{-1}$ with $|s-t|<\delta$ is 
$\mathcal{C}^1$-small enough to have a Greek generating function.
\end{proof}

\subsection{Invariants for Legendrian submanifolds}\label{s_inv_leg}

Let $B$ be a closed manifold, and denote by $\mathcal{L}$ the set of all Legendrian submanifolds of $J^1B$ contact isotopic to the 0-section. As in the symplectic case, for any $L\in\mathcal{L}$ and $u\neq0$ in $H^{\ast}(B)$ we can define a real number $c(u,L)$ by $$ c(u,L):=\text{inf}\,\{\,a\in \mathbb{R} \;|\;i_a^{\phantom{a}\ast}(u)\neq 0\,\}$$
where $i_a$ is the inclusion $(E^a,E^{-\infty})\longrightarrow (E,E^{-\infty})$ of sublevel sets of any generating function for $L$.

\begin{lemma}\label{inv_leg}
Let $\mu\in H^n(B)$ denote the orientation class of $B$. The map $H^{\ast}(B)\times \mathcal{L}\longrightarrow\mathbb{R}$, $(u,L)\longmapsto c(u,L)$ satisfies the following properties:
\vspace{-0.2cm}
\begin{enumerate}
\renewcommand{\labelenumi}{(\roman{enumi})}
\item If $L_1$, $L_2$ have generating functions $S_1$, $S_2: E \longrightarrow\mathbb{R}$ with $|S_1-S_2|_{\mathcal{C}^0}\leq \varepsilon$, then for any $u$ in $H^{\ast}(B)$ it holds that $|c(u,L_1)-c(u,L_2)|\leq \varepsilon$.
\item $$c\big(u\cup v, L_1+L_2\big)\geq c(u,L_1)+c(v,L_2) $$
where $L_1+L_2$ is defined by
$$ L_1+L_2:=\{\;(q,p,z)\in J^1B \;|\; p=p_1+p_2, \; z=z_1+z_2, $$
$$(q,p_1,z_1)\in L_1, \; (q,p_2,z_2)\in L_2  \;\}.$$
\item $$c(\mu,\overline{L})=-c(1,L),$$ where $\overline{L}$ denotes the image of $L$ under the map 
$J^1B\rightarrow J^1B$, $(q,p,z)\mapsto(q,-p,-z)$.
\item Assume $L\cap 0_B\neq \emptyset$. Then $c(\mu, L)=c(1,L)$ if and only if $L$ is the $0$-section. In this case we have $$c(\mu,L)=c(1,L)=0.$$
\end{enumerate}
\end{lemma}

\begin{proof}
If $S$ is a generating function for $L\subset J^1B$ then $S$ also generates $\pi(L)$, where $\pi$ denotes the projection $J^1B=T^{\ast}B\times\mathbb{R}\rightarrow T^{\ast}B$. So $c(u,L)=c(u,\pi(L))$ and thus all the results follow from the symplectic case. 
\end{proof}

\noindent
Property (v) of Lemma \ref{inv_lagr} does not hold in the contact case. However Bhupal \cite{B} showed that the following weaker statement is still true.

\begin{lemma}\label{questo2}
For any contactomorphism $\Psi$ of $J^1B$ contact isotopic to the identity, $u\neq 0$ in $H^{\ast}B$ and $L\in \mathcal{L}$ it holds 
$$c\big(u,\Psi(L)\big)=0 \quad\Leftrightarrow \quad c\big(u,L-\Psi^{-1}(0_B)\big)=0.$$
\end{lemma}

\begin{proof}
Let $\Psi_t$ be a contact isotopy of $J^1B$ with $\Psi=\Psi_t|_{t=1}$, and for every $t$ consider the Legendrian submanifold $\Lambda_t=\Psi_t^{\phantom{t}-1}\Psi(L)-\Psi_t^{-1}(0_B)$. We have $\Lambda_0=\Psi(L)$ and $\Lambda_1=L-\Psi^{-1}(0_B)$.
Let $c_t=c(u,\Lambda_t)$. We will prove that if $c_{t_0}=0$ for some $t\in[0,1]$ then $c_t=0$ for all $t$. Let $S_t:E\rightarrow\mathbb{R}$ be a 1-parameter family of generating functions for $\Lambda_t$. Consider a path $x_t$ in $E$ such that each $x_t$ is a critical point of $S_t$ with critical value $c_t$, for $t$ in some subinterval of $[0,1]$ containing $t_0$. Recall that $x_t$ corresponds to an intersection of $\Lambda_t$ with the 0-wall of $J^1B$. Since by hypothesis $c_{t_0}=0$, $x_{t_0}$ corresponds in fact to an intersection of $\Lambda_{t_0}$ with the 0-section. We will first assume that this intersection is transverse, so that $x_{t_0}$ is a non-degenerate critical point of $S_{t_0}$. The idea of the proof now is to construct a path $y_t$ in $E$ such that $y_{t_0}=x_{t_0}$ and each $y_t$ is a non-degenerate critical point of $S_t$ with critical value $0$. It will then follow from Morse theory that the two paths $x_t$ and $y_t$ must coincide, so that $c_t=0$ for all $t$. The path $y_t$ can be constructed as follows. The key observation is that (non-degenerate) critical points of $S_t$ with critical value $0$ are in 1-1 correspondence with (transverse) intersection points of $\Lambda_t$ with $0_B$. Moreover (transverse) intersections of $\Lambda_t$ with $0_B$ correspond to (transverse) intersections of $\Psi_t^{\phantom{t}-1}\Psi(L)$ with $\Psi_t^{\phantom{t}-1}(0_B)$ (by projecting to $0_B$), and the last correspond  to (transverse) intersections of $\Psi(L)$ with $0_B$ (by applying $\Psi_t$), i.e. of $\Lambda_0$ with $0_B$. Using this we see that $y_t':=\pi\,\Big(\Psi_t^{-1}\Psi_{t_0}\big(\widetilde{i_{S_0}(x_0)}\big)\Big)$ is a transverse intersection of $\Lambda_t$ with $0_B$, where $\widetilde{i_{S_0}(x_0)}$ denotes the point in $\Psi_{t_0}^{\phantom{t_0}-1}\Psi(L)\,\cap\,\Psi_{t_0}^{\phantom{t_0}-1}(0_B)$ that projects to $i_{S_0}(x_0)\in\Psi_{t_0}^{\phantom{t_0}-1}\Psi(L)-\Psi_{t_0}^{\phantom{t_0}-1}(0_B)$. Thus $y_t:=i_{S_t}^{\phantom{S_t}-1}(y_t')$ is the desired 1-parameter family of critical points of $S_t$. This finishes the proof under the assumption that $x_{t_0}$ is a non-degenerate critical point of $S_{t_0}$. The general case follows from an approximation argument (see \cite{B}).
\end{proof}

\noindent
In \cite{B} Bhupal realized that this result is enough to extend Viterbo's partial order to the contact case. We will review his construction in \ref{Bhupalorder}. However, Lemma \ref{questo2} is too weak to give an interesting generalizations to the contact case of the Viterbo capacity. We will now give a stronger version of Lemma \ref{questo2}, which is only available in the 1-periodic case and will enable us to define in \ref{contact_cap} a contact capacity for domains in $\mathbb{R}^{2n}\times S^1$.\\
\\
We will denote by $\lceil\cdot\rceil$ the integer part of a real number, i.e. the smallest integer that is greater or equal to the given number.

\begin{lemma}\label{questo4}
Let $\Psi$ be a contactomorphism of $J^1B$ which is 1-periodic in the $\mathbb{R}$-coordinate of $J^1B=T^{\ast}B\times\mathbb{R}$ and isotopic to the identity through 1-periodic contactomorphisms. Then for every $u\neq 0$ in $H^{\ast}(B)$ and $L\in\mathcal{L}$ it holds
$$\lceil c\big(u,\Psi(L)\big)\rceil=\lceil c\big(u,L-\Psi^{-1}(0_B)\big)\rceil.$$
\end{lemma}

\begin{proof}
Let $\Psi_t$ be a contact isotopy of $J^1B$ with $\Psi=\Psi_t|_{t=1}$, and consider $c_t=c(u,\Lambda_t)$ where $\Lambda_t=\Psi_t^{\;-1}\Psi(L)-\Psi_t^{\;-1}(0_B)$. We will show that if $k$ is an integer and $c_{t_0}=k$ for some $t_0$, then $c_t=k$ for all $t$. Let $S_t:E\rightarrow\mathbb{R}$ be a family of generating functions for $\Lambda_t$. Then $c_t$ is a critical value of $S_t$. As in the proof of Proposition \ref{questo2} the result follows if we prove that if $x_{t_0}$ is a (non-degenerate) critical point of $S_{t_0}$ with critical value $k$ then there is a 1-parameter family of (non-degenerate) critical points $y_t$ of $S_t$ with $y_{t_0}=x_{t_0}$ and all with critical value $k$. The idea to prove this is that, since the $\Psi_t$ are 1-periodic, the construction of the proof of Lemma \ref{questo2} can be adapted to the case in which the critical value $0$ is replaced by an integer $k$. More precisely, it is easy to check that if $x_{t_0}$ is a critical point of $S_{t_0}$ with critical value $k$ then 
$y_t':=\pi\,\Big(\Psi_t^{-1}\Psi_{t_0}\big(\widetilde{i_{S_0}(x_0)}\big)\Big)+(0,0,k)$ is in the intersection of $\Lambda_t$ with $0_B\times\{k\}$. Thus $y_t:=i_{S_t}^{\phantom{S_t}-1}(y_t')$ is the desired 1-parameter family of critical points of $S_t$.
\end{proof}

\subsection{Invariants for contactomorphisms of $\mathbb{R}^{2n+1}$}\label{inv_contactomorphisms}

Consider a contactomorphism $\phi$ either in $\text{Cont}_0^{\phantom{0}c}\,(\mathbb{R}^{2n+1})$ or in $\text{Cont}^{\; \; c}_{\text{1-per}} (\mathbb{R}^{2n+1})$, and define
$$c(\phi):=c(\mu,\Gamma_{\phi})$$
where $\Gamma_{\phi}$ is regarded as a Legendrian submanifold either of $J^1S^{2n+1}$ or $J^1(S^{2n}\times S^1)$ and $\mu$ is the orientation class either of $S^{2n+1}$ or $S^{2n}\times S^1$. Note that $c(\phi)$ is a critical value of any generating function for $\Gamma_{\phi}$, so by Lemma \ref{crucialcont} we have that $c(\phi)=\phi_3(q)-z$ for some translated point $q=(x,y,z)$ of $\phi$. Note also that $c(\text{id})=0$. Moreover $c$ satisfies the following properties.

\begin{prop}\label{inv_cont}
For all $\phi$, $\psi$ in $\text{Cont}_0^{\phantom{0}c}\,(\mathbb{R}^{2n+1})$ or $\text{Cont}^{\; \; c}_{\text{1-per}} (\mathbb{R}^{2n+1})$ it holds:
\vspace{-0.2cm}
\begin{enumerate}
\renewcommand{\labelenumi}{(\roman{enumi})}
\item $c(\phi)\geq 0$.
\item If $c(\phi)=c(\phi^{-1})=0$ then $\phi$ is the identity.
\item If $c(\phi)=c(\psi)=0$ then $c(\phi\psi)=0$.
\item If $\phi_1 \leq \phi_2$ in the sense of \ref{contactomorphisms} then $c(\phi_1)\leq c(\phi_2)$.
\end{enumerate}
\end{prop}

\begin{proof}
\begin{enumerate}
\renewcommand{\labelenumi}{(\roman{enumi})}
\item 
As in the symplectic case we have $c(1,\overline{\Gamma_{\phi}})\leq 0$ for all $\phi$. Thus by Lemma \ref{inv_leg}(iii) it holds that $c(\phi)=c(\mu,\Gamma_{\phi})=-c(1,\overline{\Gamma_{\phi}})\geq 0$.

\item 
Note first that, for all $u$, if $c(u,\Gamma_{\phi^{-1}})=0$ then also $c(u,\overline{\Gamma_{\phi}})=0$ (apply Lemma \ref{questo2} to $L=0_B$ and $\Psi=\Psi_{\phi^{-1}}$). Using this, the result then follows from Lemma \ref{inv_leg}(iii)(iv).

\item 
We have $c(\mu, \Psi_{\phi^{-1}}(\Gamma_{\phi\psi}))=c(\mu,\Gamma_{\psi})=0$. Thus, by Lemma \ref{questo2} and Lemma \ref{inv_leg}(ii),
$$0=c(\mu, \Gamma_{\phi\psi}-\Psi_{\phi}(0_B))=c(\mu, \Gamma_{\phi\psi}-\Gamma_{\phi})
\geq c(\mu, \Gamma_{\phi\psi}) + c(1,\overline{\Gamma_{\phi}}).$$
Since by Lemma \ref{inv_leg}(iii) it holds $c(1,\overline{\Gamma_{\phi}})=-c(\mu,\Gamma_{\phi})=0$, we have that $c(\phi\psi)=c(\mu, \Gamma_{\phi\psi})\leq 0$, and thus $c(\phi\psi)=0$.

\item
As in the symplectic case, using Proposition \ref{parordcont}.
\end{enumerate}
\end{proof}

\noindent
Using Lemma \ref{questo4} we can prove a stronger version of Proposition \ref{inv_cont}(iii), that only holds in the 1-periodic case.

\begin{prop}\label{inv_cont_per}
For all $\phi$, $\psi$ in $\text{Cont}^{\; \; c}_{\text{1-per}} (\mathbb{R}^{2n+1})$ it holds
$$\lceil c(\phi\psi)\rceil\leq \lceil c(\phi)\rceil + \lceil c(\psi)\rceil.$$
\end{prop}

\begin{proof}
We have 
$ c(\psi)=c(\mu,\Gamma_{\psi})=c\big(\mu, \Psi_{\phi^{-1}}(\Gamma_{\phi\psi})\big)$ thus by Lemma \ref{questo4} it holds
$\lceil c(\psi)\rceil=\lceil c\big(\mu,\Gamma_{\phi\psi}-\Psi_{\phi}(0_B)\big)\rceil$. But, by Lemma \ref{inv_leg}(ii)-(iii)
$$
c\big(\mu,\Gamma_{\phi\psi}-\Psi_{\phi}(0_B)\big)=
c\big(\mu\cup 1,\Gamma_{\phi\psi}-\Gamma_{\phi}\big)\geq
c\big(\mu,\Gamma_{\phi\psi}\big)+c\big(1,\overline{\Gamma_{\phi}}\big)=$$
$$
c(\phi\psi)-c\big(\mu,\Gamma_{\phi}\big)=
c(\phi\psi)-c(\phi).
$$
Thus 
$$\lceil c(\psi) \rceil \geq \lceil c(\phi\psi)-c(\phi) \rceil \geq \lceil c(\phi\psi)\rceil-\lceil c(\phi) \rceil$$
as we wanted.
\end{proof}

\noindent
In contrast with the symplectic case, $c$ is not invariant by conjugation. Recall that in the symplectic case this property follows from the fact that, for every Hamiltonian symplectomorphism $\varphi$ of $\mathbb{R}^{2n}$, $c(\varphi)$ belongs the action spectrum of $\varphi$ which is invariant by conjugation. In the contact case the situation is very different since the set of values taken by the contact action $\phi_3(q)-z$ at translated points $q=(x,y,z)$ of a contactomorphism $\phi$ of $\mathbb{R}^{2n+1}$ is not invariant by conjugation. In fact, not even the property of being a translated point is invariant by conjugation: if $q$ is a translated point for $\phi$ then in general $\psi_t(q)$ is not a translated point for $\psi_t\phi\psi_t^{\phantom{t}-1}$. However, we are going to see that this is true if $q$ is a translated point with action $0$, and in the 1-periodic case also if the action is any integer. As we will see this observation is the key to prove that, in the 1-periodic case, the integer part of $c$ is invariant by conjugation.\\
\\
Recall that a point $q$ of $\mathbb{R}^{2n+1}$ is a translated point for a contactomorphism $\phi$ if and only if $\Gamma_{\phi}(q)$ is in the intersection of $\Gamma_{\phi}$ with the 0-wall. We will say that $q$ is a \textit{non-degenerate} translated point if this intersection is transverse and thus if the corresponding critical point of the generating function of $\phi$ is non-degenerate. Note that this condition can also be expressed by requiring that there is no tangent vector $X\neq 0$ at $q$ such that $(\Gamma_{\phi})_{\ast}(X)$ is tangent to the 0-wall, or equivalently (see \cite{B}) no tangent vector $X\neq 0$ at $q$ such that $\phi_{\ast}(X)=X$ and $X(g)=0$.

\begin{lemma}[\cite{B}]\label{uffa}
Let $\phi$ and $\psi$ be contactomorphisms of $\mathbb{R}^{2n+1}$. Then $q\in\mathbb{R}^{2n+1}$ is a translated point of $\phi$ with contact action $0$ if and only if $\psi(q)$ is a translated point of $\psi\phi\psi^{-1}$ with contact action $0$. Moreover, $q$ is non-degenerate if and only if so is $\psi(q)$.
\end{lemma}

\begin{proof}
Note first that if $\phi^{\ast}(dz-ydx)=e^g(dz-ydx)$ and $\psi^{\ast}(dz-ydx)=e^f(dz-ydx)$ then
$(\psi\phi\psi^{-1})^{\ast}(dz-ydx)=e^h(dz-ydx)$ with $h=f\circ\phi\circ\psi^{-1}+g\circ\psi^{-1}-f\circ\psi^{-1}$. Suppose that $q$ is a translated point of $\phi$ with contact action $0$, i.e. $\phi(q)=q$ and $g(q)=0$. Then $\psi\phi\psi^{-1}\big(\psi(q)\big)=\psi(q)$ and $h\big(\psi(q)\big)=f\big(\phi(q)\big)+g(q)-f(q)=0$ so that $\psi(q)$ is a translated point of $\psi\phi\psi^{-1}$ with contact action $0$. To prove the last statement we will show that if $q$ is a degenerate translated point then so is $\psi(q)$. By the discussion above, if $q$ is a degenerate translated point for $\phi$ then there is a tangent vector $X\neq 0$ at $q$ such that $\phi_{\ast}(X)=X$ and $X(g)=0$. But then
$$
\big(\psi\phi\psi^{-1}\big)_{\ast}\,\big(\psi_{\ast}(X)\big)=\psi_{\ast}(X)
$$
and
\begin{eqnarray*}
\psi_{\ast}(X)(h)&=&X(f\circ\phi+g-f)=X(f\circ\phi)+X(g)-X(f)\\
&=&\phi_{\ast}(X)(f)-X(f)=0
\end{eqnarray*}
thus $\psi(q)$ is a degenerate translated point for $\psi\phi\psi^{-1}$.
\end{proof}

\noindent
We now give the 1-periodic version of the previous lemma.

\begin{lemma}\label{uffa2}
Let $\phi$ and $\psi$ be 1-periodic contactomorphisms of $\mathbb{R}^{2n+1}$, and $k$ an integer. Then $q\in\mathbb{R}^{2n+1}$ is a translated point of $\phi$ with contact action $k$ if and only if $\psi(q)$ is a translated point of $\psi\phi\psi^{-1}$ with contact action $k$. Moreover, $q$ is non-degenerate if and only if so is $\psi(q)$.
\end{lemma}

\begin{proof}
The same proof as in Lemma \ref{uffa} goes through in this situation, due to the 1-periodicity of $\psi$ and the fact that $k$ is an integer. Suppose indeed that $q$ is a translated point of $\phi$ with contact action $k$, i.e. $\phi(q)=q+(0,0,k)$ and $g(q)=0$. Then
$\psi\phi\psi^{-1}\big(\psi(q)\big)=\psi\big(\phi(q)\big)=\psi\big(q+(0,0,k)\big)=\psi(q)+(0,0,k)$
and 
$$h\big(\psi(q)\big)=f\big(\phi(q)\big)+g(q)-f(q)=f\big(q+(0,0,k)\big)+g(q)-f(q)=0$$ 
(note that $f$ is invariant by integer translation in the $z$-coordinate since $\psi$ is 1-periodic), thus $\psi(q)$ is a translated point of $\psi\phi\psi^{-1}$ with contact action $k$. The statement about the non-degeneracy can be seen as in the proof of Lemma \ref{uffa}.
\end{proof}

\noindent
The above lemma is the key to prove the following crucial result.

\begin{lemma}\label{crucial}
Consider a contactomorphism $\phi$ and a contact isotopy $\psi_t$ in $\text{Cont}^{\; \; c}_{\text{1-per}} (\mathbb{R}^{2n+1})$ and let $S_t:E\rightarrow \mathbb{R}$ be a 1-parameter family of generating functions for the conjugation $\psi_t\phi\psi_t^{\phantom{t}-1}$.
If $k$ is an integer and $c_t$ is a path of critical values of $S_t$ with $c_{t_0}=k$ for some $t_0\in\mathbb{R}$, then $c_t=k$ for all $t$.
\end{lemma}

\begin{proof}
Suppose that $c_t$ is a path of critical values of $S_t$ with $c_{t_0}=k$ for some $t_0$. Let $x_t=(q_t,\xi_t)\in\mathbb{R}^{2n+1}\times\mathbb{R}^N$ be a 1-parameter family of critical points of $S_t$, for $t$ in some subinterval of $[0,1]$ containing $t_0$. Following the model of the proof of Lemma \ref{questo2}, the result follows if we construct a path $y_t$ in $E$ such that $y_{t_0}=(q_{t_0},\xi_{t_0})$ and every $y_t$ is a (non-degenerate) critical point of $S_t$ with critical value $k$ (assuming that $x_t$ is non-degenerate). We know that $q_{t_0}$ is a non-degenerate translated point for $\psi_{t_0}\phi\psi_{t_0}^{\phantom{t_0}-1}$ with action $c_{t_0}=k$. By Lemma \ref{uffa2} it follows that $\psi_t\big(\psi_{t_0}^{\phantom{t_0}-1}(q_{t_0})\big)$ is a path of non-degenerate translated points for $\psi_t\phi\psi_t^{-1}$, all with action $k$. Thus $y_t:=i_{S_t^{\phantom{t}-1}}\Big(\psi_t\big(\psi_{t_0}^{\phantom{t_0}-1}(q_{t_0})\big),0,k\Big)$ is the desired path of critical points of $S_t$.
\end{proof}

\noindent
Lemma \ref{crucial} immediately implies that in the 1-periodic case the integer part of $c$ is invariant by conjugation, as stated in the following proposition. As we will see, this result will allow us to define in \ref{contact_cap} an integral invariant for domains in $\mathbb{R}^{2n}\times S^1$.

\begin{prop}\label{conj_per}
For any $\phi$, $\psi$ in $\text{Cont}^{\; \; c}_{\text{1-per}} (\mathbb{R}^{2n+1})$ it holds
$$
\lceil c(\phi)\rceil=\lceil c(\psi\phi\psi^{-1})\rceil.
$$
\end{prop}

\noindent
In the case of $\text{Cont}^c\,(\mathbb{R}^{2n+1})$ only the following weaker statement is true.

\begin{prop}[\cite{B}]\label{conj_cont}
For any $\phi$, $\psi$ in $\text{Cont}_0^{\phantom{0}c}\,(\mathbb{R}^{2n+1})$ we have that $c(\phi)=0$ if and only if $c(\psi\phi\psi^{-1})=0$.
\end{prop}

\begin{proof}
Let $\psi$ be the time-1 map of the contact isotopy $\psi_t$ and consider $c_t=c(\psi_t\phi\psi_t^{-1})$ . As in the proof of Lemma \ref{crucial}, Lemma \ref{uffa} implies that if $c_{t_0}=0$ then $c_t=0$ for all $t$.
\end{proof}

\noindent
We end this section explaining the relation between the invariant $c$ in the symplectic and contact case.

\begin{prop}\label{lift_c}
Let $\varphi$ be a compactly supported Hamiltonian symplectomorphism of $\mathbb{R}^{2n}$ and $\widetilde{\varphi}$ its lift to $\mathbb{R}^{2n+1}$ or to $\mathbb{R}^{2n}\times S^1$. Then $c(\widetilde{\varphi})=c(\varphi)$.
\end{prop}

\begin{proof}
The result follows from Lemma \ref{lift}. The case of $\mathbb{R}^{2n+1}$ is immediate, while the 1-periodic case can be seen as follows. Suppose that $\widetilde{\varphi}$ is the lift of $\varphi$ to $\mathbb{R}^{2n}\times S^1$. By Lemma \ref{lift} we know that a generating function for $\widetilde{\varphi}$ is given by $\widetilde{S}:(S^{2n}\times S^1)\times \mathbb{R}^N\rightarrow\mathbb{R}$, $\widetilde{S}(q,z;\xi)=S(q;\xi)$ where $S:S^{2n}\times \mathbb{R}^N\rightarrow\mathbb{R}$ is a generating function for $\varphi$. Denote by $\widetilde{E}^a$ the sublevel set of $\widetilde{S}$ with respect to $a$, and by $\widetilde{i_a}$ the inclusion $(\widetilde{E}^a,\widetilde{E}^{-\infty})\hookrightarrow(\widetilde{E},\widetilde{E}^{-\infty})$. Then $\widetilde{E}^a=E^a\times S^1$ and, after identifying $H^{\ast}(\widetilde{E},\widetilde{E}^{-\infty})$ with $H^{\ast}(S^{2n}\times S^1)=H^{\ast}(S^{2n})\otimes H^{\ast}(S^1)$ and $H^{\ast}(\widetilde{E}^a,\widetilde{E}^{-\infty})$ with $H^{\ast}(E^a,E^{-\infty})\otimes H^{\ast}(S^1)$, the induced map
$$
\widetilde{i_a}^{\;\ast}: H^{\ast}(S^{2n})\otimes H^{\ast}(S^1)\rightarrow H^{\ast}(E^a,E^{-\infty})\otimes H^{\ast}(S^1)
$$
is given by $\widetilde{i_a}^{\;\ast}=i_a^{\phantom{a}\ast}\otimes \text{id}$. In particular we have that $\widetilde{i_a}^{\;\ast}(\mu \otimes\mu_{S^1})=i_a^{\phantom{a}\ast}(\mu)\otimes\mu_{S^1}$ where $\mu$ and $\mu_{S^1}$ denote respectively the orientation classes of $S^{2n}$ and $S^1$, thus $\widetilde{i_a}^{\;\ast}(\mu \otimes\mu_{S^1})=0$ if and only if 
$i_a^{\phantom{a}\ast}(\mu)=0$. Since $\mu \otimes\mu_{S^1}$ is the orientation class of $H^{\ast}(S^{2n}\times S^1)$ we conclude that $c(\widetilde{\varphi})=c(\varphi)$.
\end{proof}

\subsection{The Bhupal partial order on $\text{Cont}_0^{\;c}\,(\mathbb{R}^{2n+1})$ and $\text{Cont}_0^{\;c}\,(\mathbb{R}^{2n}\times S^1)$
} \label{Bhupalorder}

Bhupal's partial order $\leq_B$ on $\text{Cont}_0^{\phantom{0}c}\,(\mathbb{R}^{2n+1})$ and on $\text{Cont}_0^{\phantom{0}c}\,(\mathbb{R}^{2n}\times S^1)$ is defined by 
$$ \phi_1 \leq_B \phi_2 \quad \text{if} \quad c(\phi_1\phi_2^{\phantom{2}-1})=0.$$
Using the properties in Proposition \ref{inv_cont} it is immediate to see that $\leq_B$ is indeed a partial order, that it is bi-invariant (i.e. if $\phi_1 \leq_B \phi_2$ and $\psi_1 \leq_B \psi_2$ then $\phi_1\psi_1 \leq_B \phi_2\psi_2$), and that if $\phi_1 \leq \phi_2$ in the sense of \ref{contactomorphisms} then $\phi_1 \leq_B \phi_2$. In particular it follows that $\leq$ is also a partial order. Note that in the language of \cite{EP} this means that $\mathbb{R}^{2n+1}$ and $\mathbb{R}^{2n}\times S^1$ are \textit{orderable} contact manifolds.

\subsection{Contact capacity of domains in $\mathbb{R}^{2n}\times S^1$}\label{contact_cap}

We will consider domains in $\mathbb{R}^{2n}\times S^1$ as domains in $\mathbb{R}^{2n+1}$ that are invariant by the action of $\mathbb{Z}$ by translations in the $z$-coordinate. For an open and bounded domain $\mathcal{V}$ of $\mathbb{R}^{2n}\times S^1$ we define the \textbf{contact capacity} of $\mathcal{V}$ as
$$c(\mathcal{V}):=\text{sup}\,\{\,\lceil c(\phi)\rceil \;|\; \phi\in\text{Cont}\,(\mathcal{V})\,\}$$ 
where $\text{Cont}\,(\mathcal{V})$ denotes the set of time-1 maps of 1-periodic contact Hamiltonian functions supported in $\mathcal{V}$. By the following lemma, $c(\mathcal{V})$ is a well-defined integer number. 

\begin{lemma}\label{questo5_c}
For every contactomorphism $\psi$ in $\text{Cont}^{\; \; c}_{\text{1-per}} (\mathbb{R}^{2n+1})$ such that $\psi(\mathcal{V})\cap\mathcal{V}= \emptyset$ we have $c(\mathcal{V})\leq \gamma(\psi)$, where $\gamma(\psi):=\lceil c(\psi)\rceil +\lceil c(\psi^{-1})\rceil$.
\end{lemma}

\begin{proof}
We will show that $\lceil c(\psi\phi)\rceil=\lceil c(\psi)\rceil$ for all $\phi$ in $\text{Cont}\,(\mathcal{V})$ and $\psi$ as in the statement of the lemma, and then conclude as in the proof of Lemma \ref{questo5}. Let $\phi=\phi_t|_{t=1}$, and consider the map $t \mapsto c(\psi\phi_t)$. Suppose $c_{t_0}=k\in\mathbb{Z}$. Then there is a translated point $q=(x,y,z)$ of $\psi\phi_{t_0}$ such that $(\psi\phi_{t_0})_3-z=k$. But then we can apply an argument similar to the one in Lemma \ref{questo5} to see that $q$ is also an almost fixed point of $\psi\phi_t$ for all $t$, with $(\psi\phi_t)_3-z=k$. We can now conclude, as in Lemma \ref{crucial}, that $c(\psi\phi_t)=k$ for all $t$. It follows that $\lceil c(\psi\phi_t)\rceil$ is independent of $t$, in particular $\lceil c(\psi\phi)\rceil=\lceil c(\psi)\rceil$.
\end{proof}

\noindent
As in the symplectic case, we can extend the definition to arbitrary domains of $\mathbb{R}^{2n}\times S^1$.

\begin{thm}\label{u}
$c$ satisfies the following properties:
\vspace{-0.2cm}
\begin{enumerate}
\renewcommand{\labelenumi}{(\roman{enumi})}
\item (Contact Invariance) For any $\psi$ in $\text{Cont}_0^{\phantom{0}c}\,(\mathbb{R}^{2n}\times S^1)$ we have $c(\psi(\mathcal{V}))=c(\mathcal{V})$.
\item (Monotonicity) If $\mathcal{V}_1\subset\mathcal{V}_2$, then $c(\mathcal{V}_1)\leq c(\mathcal{V}_2)$.
\item For any domain $\mathcal{U}$ in $\mathbb{R}^{2n}$ we have $c\,(\mathcal{U}\times S^1)=\lceil c(\mathcal{U})\rceil$.
\end{enumerate}
\end{thm}

\begin{proof}
Contact invariance follows from Proposition \ref{conj_per}, and monotonicity is immediate from the definition. As for the last property, it can be seen as follows. If $\varphi$ is an Hamiltonian symplectomorphism of $\mathbb{R}^{2n}$ generated by a Hamiltonian $H:\mathbb{R}^{2n}\rightarrow \mathbb{R}$ supported in $\mathcal{U}$, then its lift $\widetilde{\varphi}$ is generated by the contact Hamiltonian $\widetilde{H}:\mathbb{R}^{2n}\times S^1\rightarrow \mathbb{R}$, $\widetilde{H}(x,y,z)=H(x.y)$ which is supported in $\mathcal{U}\times S^1$. By Proposition \ref{lift_c} we have $c(\widetilde{\varphi})=c(\varphi)$, so we see that $c\,(\mathcal{U}\times S^1)\geq\lceil c(\mathcal{U})\rceil$. Equality holds because for every $\phi$ in $\text{Cont}\,(\mathcal{U}\times S^1)$ there exists a $\varphi$ in $\text{Ham}\,(\mathcal{U})$ such that $\phi \leq \widetilde{\varphi}$.
\end{proof}

Note that the Non-Squeezing Theorem of Eliashberg, Kim and Polterovich follows immediately from Theorem \ref{u} and Example \ref{cap_ell}. Indeed, consider $R_2\leq k < R_1$ for $k \in \mathbb{Z}$ and suppose that there is a contactomorphism $\psi$ in $\text{Cont}_0^{\phantom{0}c}\,(\mathbb{R}^{2n}\times S^1)$ such that $\psi\,\big(\widehat{B(R_1)}\big)\subset\widehat{B(R_2)}$. Then by monotonicity we have $c\,\Big(\psi\,\big(\widehat{B(R_1)}\big)\Big)\leq c\,\big(\widehat{B(R_2)}\big)$. But this is impossible since $c\,\Big(\psi\,\big(\widehat{B(R_1)}\big)\Big)=c\,\big(\widehat{B(R_1)}\big)=\lceil c\,\big(B(R_1)\big)\rceil>k$
and $c\,\big(\widehat{B(R_2)}\big)=\lceil c\,\big(B(R_2)\big)\rceil\leq k$. Note that the same argument shows that if $R_2\leq k < R_1$ it is in fact not even possible to squeeze $\widehat{B(R_1)}$ into $\widehat{C(R_2)}$.

\subsection{Contact homology of domains in $\mathbb{R}^{2n}\times S^1$}\label{2.3}

In this last section we generalize to the contact case Traynor's construction of symplectic homology. Similarly to the case of the capacity, we only obtain contact invariant homology groups $G_{\ast}^{\;\;a,b}\,(\mathcal{V})$ for domains $\mathcal{V}$ in $\mathbb{R}^{2n}\times S^1$ and for \textit{integer} parameters $a$ and $b$.\\
\\
Let $\phi$ be a contactomorphism in $\text{Cont}^{\; \; c}_{\text{1-per}} (\mathbb{R}^{2n+1})$ with generating function $S:E=(S^{2n}\times S^1)\times\mathbb{R}^N\longrightarrow\mathbb{R}$. Given integer numbers $a$ and $b$ that are not critical values of $S$ and such that $-\infty<a<b\leq\infty$, we define the \textbf{$k$-th contact homology group} of $\phi$ with respect to the values $a$ and $b$ by
$$G_k^{\;\;(a,b]}\,(\phi):=H_{k+\iota}\,(E^b, E^a)$$
where $E^a$, $E^b$ denote the sublevel sets of $S$, and $\iota$ is the index of the quadratic at infinity part of $S$. By the uniqueness part in Theorem \ref{eugfcont} these groups are well-defined, i.e. do not depend on the choice of $S$.\\
\\
The following proposition follows immediately
from Lemma \ref{lift}.

\begin{prop}\label{lift2}
For any $\varphi$ in $\text{Ham}^c(\mathbb{R}^{2n})$ we have 
$$G_{\ast}^{\;\;(a,b]}\,(\widetilde{\varphi})=G_{\ast}^{\;\;(a,b]}\,(\varphi)\otimes H_{\ast}(S^1).$$
\end{prop}

The definition of $G_k^{\;\;(a,b]}\,(\phi)$ would in fact make sense for all real numbers $a$ and $b$ and also for contactomorphisms of $\text{Cont}^c\,(\mathbb{R}^{2n+1})$. However, the facts that $a$ and $b$ are integers and $\phi$ is 1-periodic are crucial to prove the following proposition.

\begin{prop}\label{conjcont}
For any $\phi$ and $\psi$ in $\text{Cont}^{\;\;c}_{\mbox{\begin{scriptsize}1-per\end{scriptsize}}}(\mathbb{R}^{2n+1})$ we have an induced isomorphism 
$$\psi_{\ast}: G_{\ast}^{\;\;(a,b]}\,(\psi\phi\psi^{-1})\longrightarrow G_{\ast}^{\;\;(a,b]}\,(\phi).$$
\end{prop}

\begin{proof}
Let $\psi$ be  the time-1 map of an isotopy $\psi_t$ of 1-periodic contactomorphisms of $\mathbb{R}^{2n+1}$, and let $S_t:\mathbb{R}^{2n+1}\times\mathbb{R}^N\longrightarrow\mathbb{R}$ be generating functions for $\psi_t\phi\psi_t^{-1}$. In contrast to the symplectic case the critical values of $S_t$ are not fixed. However we will now see that, due to Lemma \ref{crucial}, we can still find an isotopy conjugating the preimages $S_t^{\phantom{t}-1}(a)$ and $S_t^{\phantom{t}-1}(b)$.
Recall that $G_{\ast}^{\;\;(a,b]}\,(\phi)$ is only defined in the case that $a$ and $b$ are not critical values of the generating function $S_0$ of $\phi$. Since $a$ and $b$ are integers, it follows from Lemma \ref{crucial} that $a$ and $b$ are not critical values of $S_t$, for any $t$. Thus we can apply an analogue of Lemma \ref{pseudomorse} for pairs of sublevel sets to find an isotopy $\theta_t$ of $\mathbb{R}^{2n+1}\times\mathbb{R}^N$ such that
$\theta_t\,\big(S_0^{\phantom{0}-1}((\infty,a])\big)=S_t^{\phantom{t}-1}((\infty,a])$ and $\theta_t\,\big(S_0^{\phantom{0}-1}((\infty,b])\big)=S_t^{\phantom{t}-1}((\infty,b])$. In particular for $t=1$ this induces the desired isomorphism $\psi_{\ast}: G_{\ast}^{\;\;(a,b]}\,(\psi\phi\psi^{-1})\longrightarrow G_{\ast}^{\;\;(a,b]}\,(\phi)$.
\end{proof}

\noindent
Consider now a domain $\mathcal{V}$ in $\mathbb{R}^{2n}\times S^1$. Given integer numbers $a$ and $b$, we denote by $\text{Cont}_{a,b}^{\;\;\;\;\;c}\,(\mathcal{V})$ the set of $\phi$ in $\text{Cont}^{\; \; c}_{\text{1-per}} (\mathbb{R}^{2n+1})$ with support contained in $\mathcal{V}$ and whose generating function does not have $a$, $b$ as critical values. Note that $\text{Cont}_{a,b}^{\;\;\;\;\;c}\,(\mathcal{V})$ is directed with respect to the partial order $\leq$ defined by the Hamiltonians, i.e. for any $\phi$, $\psi$ in $\text{Cont}_{a,b}^{\;\;\;\;\;c}\,(\mathcal{V})$ there is a $\varphi$ in $\text{Cont}_{a,b}^{\;\;\;\;\;c}\,(\mathcal{V})$ such that $\phi\leq\varphi$ and $\psi\leq\varphi$.
Suppose now that $\phi_1\leq\phi_2$. Then by Proposition \ref{parordcont} we know that there are generating functions $S_1$, $S_2: E \longrightarrow \mathbb{R}$ for $\Gamma_{\phi_1}$, $\Gamma_{\phi_2}$ respectively such that $S_1\leq S_2$. Thus we have inclusions of sublevel sets $E_2^{\;a}\subset E_1^{\;a}$ and $E_2^{\;b}\subset E_1^{\;b}$, and so an induced homomorphism $\lambda_1^{\phantom{1}2}: G_k^{\;\;(a,b]}\,(\phi_2)\longrightarrow G_k^{\;\;(a,b]}\,(\phi_1)$. Note that given $\phi_1$, $\phi_2$, $\phi_3$ in $\text{Cont}_{a,b}^{\;\;\;\;\;c}\,(\mathcal{V})$ with $\phi_1\leq\phi_2\leq\phi_3$, it holds $\lambda_3^{\phantom{3}2}\circ \lambda_2^{\phantom{2}1}= \lambda_3^{\phantom{3}1}$ and $\lambda_i^{\phantom{i}i}=\text{id}$. This means in particular that $\lbrace G_k^{\;\;(a,b]}\,(\phi_i)\rbrace_{\phi_i\in\text{Cont}_{a,b}^{\;\;\;\;\;c}\,(\mathcal{V})}$ is an inversely directed family of groups, so we can define the \textbf{$k$-th contact homology group} $G_k^{\;\;(a,b]}\,(\mathcal{V})$ of $\mathcal{V}$ with respect to the values $a$ an $b$ to be the inverse limit of this family. Note that $G_k^{\;\;a,b}\,(\mathcal{V})$ can be calculated by any sequence $\phi_1\leq\phi_2\leq\phi_3\leq \cdots$ such that the associated contact Hamiltonians get arbitrarily large.\\
\\
The next two theorems are proved as in the symplectic case (using Proposition \ref{conjcont} for the first).

\begin{thm}[Contact invariance]\label{cont_inv}
For any domain $\mathcal{V}$ in $\mathbb{R}^{2n}\times S^1$ and any contactomorphism $\psi$ of $\mathbb{R}^{2n}\times S^1$ isotopic to the identity we have an induced isomorphism 
$\psi_{\ast}:G_k^{\;\;(a,b]}\,\big(\psi(\mathcal{V})\big)\longrightarrow G_k^{\;\;(a,b]}\,(\mathcal{V})$.
\end{thm}

\begin{thm}[Monotonicity]\label{monot}
Every inclusion of domains induces a homomorphism of homology groups (reversing the order), with the following functorial properties: 
\begin{enumerate}
\renewcommand{\labelenumi}{(\roman{enumi})}
\item
If $\mathcal{V}_1\subset\mathcal{V}_2\subset\mathcal{V}_3$ then the following diagram commutes
\begin{displaymath}
\xymatrix{
 G_{\ast}^{\;\;(a,b]}\,(\mathcal{V}_3) \ar[r] \ar[rd] &
 G_{\ast}^{\;\;(a,b]}\,(\mathcal{V}_2) \ar[d] \\
 & G_{\ast}^{\;\;(a,b]}\,(\mathcal{V}_1).}
\end{displaymath} 
\item
If $\mathcal{V}_1\subset\mathcal{V}_2$, then for any contactomorphism $\psi$ the following diagram commutes
\begin{displaymath}
\xymatrix{
 G_{\ast}^{\;\;(a,b]}\,(\mathcal{V}_2) \ar[r] &
 G_{\ast}^{\;\;(a,b]}\,(\mathcal{V}_1)  \\
 G_{\ast}^{\;\;(a,b]}\,\big(\psi(\mathcal{V}_2)\big) \ar[u]^{\psi_{\ast}} \ar[r] &  
 G_{\ast}^{\;\;(a,b]}\,\big(\psi(\mathcal{V}_1)\big) \ar[u]_{\psi_{\ast}}.}
\end{displaymath} 
\end{enumerate}
\end{thm}

\noindent
The relation between symplectic and contact homology is given by the following theorem.


\begin{thm}\label{lift3}
For any domain $\mathcal{U}$ of $\mathbb{R}^{2n}$ we have $G_{\ast}^{\;\;(a,b]}\,(\mathcal{U}\times S^1)=G_{\ast}^{\;\;(a,b]}\,(\mathcal{U})\otimes H_{\ast}(S^1)$. Moreover, this correspondence is functorial in the following sense. Let $\mathcal{U}_1$, $\mathcal{U}_2$ be domains in $\mathbb{R}^{2n}$ with $\mathcal{U}_1\subset\mathcal{U}_2$, and for $i=1,2$ identify $G_{\ast}^{\;\;(a,b]}\,(\mathcal{U}_i\times S^1)$ with $G_{\ast}^{\;\;(a,b]}\,(\mathcal{U}_i)\otimes H_{\ast}(S^1)$. Then the homomorphism $G_{\ast}^{\;\;(a,b]}\,(\mathcal{U}_2\times S^1)\rightarrow G_{\ast}^{\;\;(a,b]}\,(\mathcal{U}_1\times S^1)$ induced by the inclusion $\mathcal{U}_1\times S^1\hookrightarrow\mathcal{U}_2\times S^1$ is given by $\mu\otimes \text{id}$, where $\mu: G_{\ast}^{\;\;(a,b]}\,(\mathcal{U}_2)\rightarrow G_{\ast}^{\;\;(a,b]}\,(\mathcal{U}_1)$ is the homomorphism induced by $\mathcal{U}_1\hookrightarrow\mathcal{U}_2$.
\end{thm}

\begin{proof}
If $\varphi_1\leq\varphi_2\leq\varphi_3\leq \cdots$ is an unbounded ordered sequence supported in $\mathcal{U}$ then $\widetilde{\varphi_1}\leq\widetilde{\varphi_2}\leq\widetilde{\varphi_3}\leq \cdots$ in an unbounded ordered sequence supported in $\mathcal{U}\times S^1$, thus the first statement follows from Proposition \ref{lift2}. Suppose now that $\mathcal{U}_1\subset\mathcal{U}_2$, and consider unbounded ordered sequences $\varphi_1^{\phantom{1}1}\leq\varphi_2^{\phantom{2}1}\leq\varphi_3^{\phantom{3}1}\leq \cdots$ and $\varphi_1^{\phantom{1}2}\leq\varphi_2^{\phantom{2}2}\leq\varphi_3^{\phantom{3}2}\leq \cdots$ supported in $\mathcal{U}_1$ and $\mathcal{U}_2$ respectively and such that $\varphi_i^{\phantom{i}1}\leq\varphi_i^{\phantom{i}2}$. Then the homomorphism $G_{\ast}^{\;\;(a,b]}\,(\mathcal{U}_2)\rightarrow G_{\ast}^{\;\;(a,b]}\,(\mathcal{U}_1)$ is induced by the homomorphisms $G_{\ast}^{\;\;(a,b]}\,(\varphi_i^{\phantom{i}2})\rightarrow G_{\ast}^{\;\;(a,b]}\,(\varphi_i^{\phantom{i}1})$. If we calculate the contact homology of $\mathcal{U}_1\times S^1$ and $\mathcal{U}_2\times S^1$ using the sequences $\widetilde{\varphi_1^{\phantom{1}1}}\leq\widetilde{\varphi_2^{\phantom{2}1}}\leq\widetilde{\varphi_3^{\phantom{3}1}}\leq \cdots$ and $\widetilde{\varphi_1^{\phantom{1}2}}\leq\widetilde{\varphi_2^{\phantom{2}2}}\leq\widetilde{\varphi_3^{\phantom{3}2}}\leq \cdots$ then the homomorphism $G_{\ast}^{\;\;(a,b]}\,(\mathcal{U}_2\times S^1)\rightarrow G_{\ast}^{\;\;(a,b]}\,(\mathcal{U}_1\times S^1)$ is induced by the homomorphisms $G_{\ast}^{\;\;(a,b]}\,(\widetilde{\varphi_i^{\phantom{i}2}})=G_{\ast}^{\;\;(a,b]}\,(\varphi_i^{\phantom{i}2})\otimes H_{\ast}(S^1)\rightarrow G_{\ast}^{\;\;(a,b]}\,(\widetilde{\varphi_i^{\phantom{i}1}})=G_{\ast}^{\;\;(a,b]}\,(\varphi_i^{\phantom{i}1})\otimes H_{\ast}(S^1)$ which are obtained by tensoring $G_{\ast}^{\;\;(a,b]}\,(\varphi_i^{\phantom{i}2})\rightarrow G_{\ast}^{\;\;(a,b]}\,(\varphi_i^{\phantom{i}1})$ with the identity on $H_{\ast}(S^1)$.
\end{proof}

\end{document}